\documentclass[a4paper, 11pt, reqno]{amsart}

\usepackage{xcolor}

\numberwithin{equation}{section}

\usepackage{amssymb}
\usepackage{mathrsfs}
\usepackage{dsfont}
\usepackage{stmaryrd, MnSymbol}
\usepackage{enumerate, xspace}
\usepackage{changebar}

\newtheorem{theorem}{Theorem}
\newtheorem{lemma}{Lemma}[section]

\newtheorem{remark}{Remark}[section]

\numberwithin{equation}{section}

\renewcommand{\a}{\alpha}
\renewcommand{\b}{\beta}

\def\R{{\mathbb{R}}}
\def\N{{\mathbb{N}}}
\def\Z{{\mathbb{Z}}}
\def\T{{\mathbb{T}}}
\def\C{{\mathbb{C}}}

\begin{document}
\title{$5/6$-Superdiffusion of energy for coupled charged harmonic oscillators in a magnetic field}

\author[K.~Saito]{Keiji Saito}
\address{Department of Physics, Faculty of Science and Technology, Keio University, 3-14-1, Hiyoshi, Kohoku-ku, Yokohama, Kanagawa, 223-8522, Japan}
\email{saitoh@rk.phys.keio.ac.jp}

\author[M.~Sasada]{Makiko Sasada}
\address{Graduate School of Mathematical Sciences, University of Tokyo, 3-8-1, Komaba, Meguro-ku, Tokyo, 153--8914, Japan}
\email{sasada@ms.u-tokyo.ac.jp}

\author{Hayate Suda}
\address{Graduate School of Mathematical Sciences, University of Tokyo, 3-8-1, Komaba, Meguro-ku, Tokyo, 153--8914, Japan}
\email{hayates@ms.u-tokyo.ac.jp}

\begin{abstract}
We consider a one-dimensional infinite chain of coupled charged harmonic oscillators in a magnetic field with a small stochastic perturbation of order $\epsilon$. We prove that for a space-time scale of order $\epsilon^{-1}$ the density of energy distribution (Wigner distribution) evolves according to a linear phonon Boltzmann equation. We also prove that an appropriately scaled limit of solutions of the linear phonon Boltzmann equation is a solution of the fractional diffusion equation with exponent $5/6$. 
\end{abstract}

\maketitle

\section{Introduction}

There has been much progress during the past decades in the understanding of superdiffusion in one dimensional systems with several conservation laws. Chains of coupled oscillators are typical models showing superdiffusive transport of energy. They are the one-dimensional Hamiltonian systems
\begin{align*} 
\begin{cases}
\frac{d}{dt} q_x(t) & = \partial_{v_x}\mathcal{H}(v_x(t),q_x(t)) \\
\frac{d}{dt} v_x(t) & = -\partial_{q_x}\mathcal{H}(v_x(t),q_x(t)) , 
\end{cases}
\end{align*}
with Hamiltonian 
\begin{align*}
&\mathcal{H} =  \sum_{x \in \Z} \left( \frac{|v_x|^{2}}{2}  + V(q_x- q_{x+1})  \right). 
\end{align*}
Here $v_x(t)$ is the velocity of the oscillator $x$ at time $t$ and $q_x(t)$ is the displacement from its equilibrium position of the oscillator $x$ at time $t$.
In the case where the potential $V$ is quadratic, the dynamics is linear and the chain is said to be harmonic and otherwise anharmonic. The Fermi-Pasta-Ulam chain (FPU chain) has possibly cubic and/or quartic terms in the potential. Super diffusion of energy and the divergence of the corresponding thermal conductivity have been observed numerically in the dynamics of FPU chains (\cite{D},\cite{Ls},\cite{LLP}). Strong efforts are made to identify the exponent of the divergence and the nature of superdiffusion in FPU chains numerically and theoretically in recent years.

In an innovative article \cite{S}, Spohn discussed an asymptotic behavior of time-dependent correlation functions of heat mode applying the method of fluctuating hydrodynamics. His argument suggests that for general anharmonic chains the macroscopic diffusion of energy is governed by the fractional diffusion equation
\begin{align}\label{fractional}
\partial_{t} \mathbf{e}(y,t) = - (- \Delta_{y})^{\frac{s}{2}}\mathbf{e}(y,t).
\end{align}
Moreover, Spohn's theory suggests that there are only two universality classes, $s = \frac{3}{2}$ or $\frac{5}{3}$.

However, a rigorous mathematical analysis of the energy transport in the anharmonic chains is too hard to justify Spohn's theory. Recently as an analytically tractable model, the \textit{harmonic} chains of oscillators with a stochastic exchange of momentum between neighboring sites, which we call the momentum exchange model, was introduced \cite{BBO}. In \cite{BBO} the authors prove the divergence of the thermal conductivity for this model and obtain an explicit exponent of the divergence of Green-Kubo formula. To understand the nature of superdiffusion for this model, a weak noise limit is studied in \cite{BOS}. They show that in the weak noise limit the time evolution of the local density of the energy is governed by the Boltzmann equation 
\begin{align}\label{BBOBoltzmann}
\partial_{t} u(y,k,t) + \frac{1}{2 \pi} \omega'(k) \partial_{y} u(y,k,t) = (\mathcal{L}u)(y,k,t) , \\
(\mathcal{L}u)(y,k,t) = \int_{\T} dk' ~ R(k,k') (u(y,k',t) - u(y,k,t)). \notag 
\end{align}
Here, the local density of energy $u(y,k,t)$ depends on the position $y \in \R$ along the chain, the wave number $k \in \T = [-\frac{1}{2},\frac{1}{2})$ and time $t \ge 0$. ${\omega}(k)$ is the dispersion relation. Later in \cite{KJO}, it is shown that a properly scaled solution of the Boltzmann equation \eqref{BBOBoltzmann} converges to the solution of the fractional diffusion equation \eqref{fractional} with $s=\frac{3}{2}$. The main idea of the proof of this convergence is the following: Since the scattering kernel $R(k,k')$ is positive, \eqref{BBOBoltzmann} can be interpreted as the forward equation for the probability density of a Markov process $(z(t),k(t))$ on $\R \times \T$. Applying a limit theorem for additive functionals of Markov processes, the scaled process $N^{ - \frac{2}{3}} z(Nt)$ converges to a L\'{e}vy process generated by $- (-\Delta)^{\frac{3}{4}}$ (up to a constant). By this two-step scaling limit, the $3/4$-fractional diffusion equation is derived from the momentum exchange model rigorously. Recently the $3/4$-fractional diffusion equation is derived by a direct limit (namely one-step scaling limit) in \cite{JKO}. For a variant of the momentum exchange model, a skew $3/4$-fractional diffusion equation is derived by a direct space-time scaling limit in \cite{BGJ}.

Most recently in \cite{SS,TSS} two of the authors introduced another variant of the momentum exchange model which also shows the superdiffusive behavior of the energy but the exponent of the divergence of Green-Kubo formula is different from the original one. The model is a chain of coupled charged harmonic oscillators in a magnetic field with a stochastic exchange of velocity between neighboring sites. 

The goal of the present paper is to understand the nature of the superdiffusion for this coupled charged harmonic chain of oscillators in a magnetic field with noise. We apply the two-step scaling limits. 
Following the idea of \cite{BOS}, we first show as Theorem \ref{thm:main} that in the weak noise limit the local density of energy is governed by the phonon linear Boltzmann equation 
\begin{align}\label{ourBoltzmann}
\partial_{t} u(y,k,i,t) & + \frac{1}{2\pi} \omega'(k) \partial_{y} u(y,k,i,t) = \mathcal{L} u(y,k,i,t) , \\
\mathcal{L}u(y,k,i,t) & = \sum_{j= 1,2} \int_{\mathbf{T}} dk' ~ R(k,i,k',j) (u(y,k',j,t)-u(y,k,i,t) ). \notag 
\end{align}
Here, the local density of energy $u(y,k,i,t)$ depends on position $y \in \R$ along the chain, the wave number $k \in \T$, the type of phonon $i =1,2$ and time $t \ge 0$. Then, we consider a properly scaled solution of the Boltzmann equation \eqref{ourBoltzmann} and show that it converges to the solution of the fractional diffusion equation \eqref{fractional} with $s=\frac{5}{3}$ as Theorem \ref{thm:main2}. This provides a first rigorous example of the $5/6$-superdiffusion of energy in a chain of oscillators.

A key ingredient of the proof of Theorem \ref{thm:main2} is the scaling limit of an additive functional of a Markov process as the prior work. Actually, since the scattering kernel $R(k,i,k',j)$ is positive, $(\ref{ourBoltzmann})$ can be interpreted as the time evolution of the density for a Markov process $(Z(t),K(t),I(t))$ on $\R \times \T \times \{ 1,2 \}$. Applying a general limit theorem in \cite{KJO}, we show that the scaled process $N^{ - \frac{3}{5}} Z(Nt)$ converges to a L\'{e}vy process generated by $- (-\Delta)^{\frac{5}{6}}$ (up to a constant) as Theorem \ref{thm:main3}. 

The difference of the exponents between $\frac{3}{4}$ (obtained in \cite{KJO,JKO} for the original momentum exchange model) and $\frac{5}{6}$ is explained by the asymptotic behavior of the derivative of the dispersion relation $\omega'(k)$ and the mean value of the scattering kernel $R(k) = \int_{\T} R(k,k')dk'$ as $k \to 0$. (We abbreviate the term $i,j$.) Roughly speaking, if 
\begin{align*}
\omega'(k) \sim k^{a}, ~ R(k) \sim k^{b}  ~ \textit{as}  ~ k \to 0
\end{align*}
for some $a,b \in \N_{\ge 0}$, by applying the argument in \cite{KJO} formally, one will obtain a L\'{e}vy process generated by $- (- \Delta)^{\frac{b+1}{2(b-a)}}$ as a proper scaling limit if $0 < \frac{b+1}{2(b-a)} < 1$ and by $\Delta$ if $\frac{b+1}{2(b-a)} \ge 1$. For the original momentum exchange model presented in \cite{BOS} and \cite{JKO} 
\begin{align*}
\omega'(k) \sim 1, ~ R(k) \sim k^{2} ~ \textit{as}  ~  k \to 0,
\end{align*}
while in our model
\begin{align*}
\omega'(k) \sim k, ~ R(k) \sim k^{4} ~ \textit{as}~  k \to 0.
\end{align*}
In particular, our model has the vanishing sound speed since $\lim_{k \to 0}\omega'(k)=0$. To be more precise, in our model $R(k,i) = \sum_{j=1}^2\int_{\T} R(k,i,k',j)dk'$ satisfies $R(k,1) \sim k^{2}$ and $R(k,2) \sim k^{4}$ (or $R(k,2) \sim k^{2}$ and $R(k,1) \sim k^{4}$ depending on the sign of the magnetic field) and the latter dominates the macroscopic evolution. Note that for a class of non-acoustic chains introduced in \cite{KO}, 
\begin{align*}
\omega'(k) \sim k, ~ R(k) \sim k^{2} ~ \textit{as} ~  k \to 0
\end{align*}
and so its macroscopic evolution is diffusive.


A technically crucial idea of our proof of Theorem \ref{thm:main} is that we consider the microscopic local density of energy, called the Wigner distribution in physics, associated to the eigenvectors of the deterministic dynamics including the effect of the magnetic field. If we employ the classical wave functions which are the eigenvectors of the harmonic Hamiltonian dynamics (without a magnetic field) and study its associated Wigner distribution, then we obtain a \textit{system} of Boltzmann equations as the weak noise limit. However, so far we do not know how to rescale the solutions of the system and derive the fractional diffusion equation from it. By employing the modified wave functions, instead of the classical wave functions, we obtain a \textit{single} limiting Boltzmann equation which is much easier to analyze. This strategy can be applied to derive the limiting equation from other Hamiltonian systems with some energy-conservative external field.  

Our paper is organized as follows: In Section 2 we prepare some notations. In Section 3 we introduce our model, wave functions and its associated Wigner distribution. Note that since we consider the infinite system, we need to define our model in terms of wave functions to make the argument rigorous. In Section 4 we state our main results, Theorem \ref{thm:main} and \ref{thm:main2}. We study a Markov process associated to our Boltzmann equation and its scaling limit in Section 5. Proofs of Theorem \ref{thm:main} and \ref{thm:main2} are given in Sections 6 and 7 respectively. 

\section{Notations}

Let $\T \cong [-\frac{1}{2},\frac{1}{2})$ be the one-dimensional torus. For $f \in \ell^{2}(\Z)$, we introduce the discrete Laplacian $\Delta f : \Z \to \R$ defined by
\begin{align*}
\Delta f (x) &= f(x+1) + f(x-1) - 2 f(x)
\end{align*}
and its Fourier transform $\widehat{f} \in L^{2}(\T) $ defined by
\begin{align*}
\widehat{f}(k) &= \sum_{x \in \Z} e^{-2 \pi \sqrt{-1} k x} f(x).
\end{align*}

For functions $f,g \in \ell^{2}(\Z)$, the discrete convolution $f * g :\Z \to \R$ is defined by
\begin{align*}
 f *g(x) &= \sum_{z \in \Z} f(x-z)g(z).
\end{align*}

For $J : \R \times \T \to \mathbb{C}$ such that $J(y,k)$ is rapidly decreasing in $y \in \R$, we define $\widehat{J} : \R \times \T \to \mathbb{C}$ as
\begin{align*}
\widehat{J}(p,k) = \int_{\R} dy ~ e^{-2 \pi \sqrt{-1} p y} J(y,k).
\end{align*}
Let $\mathbf{S}$ be the space of rapidly decreasing functions on $ \R \times \T$ defined by
\begin{align*}
\mathbf{S} =\{ J \in  C^{\infty}( \R \times \T ,\mathbb{C}) \ ; \  |J|_{m,n}  < \infty  \quad \forall m, n \in \Z_{\ge 0} \}
\end{align*}
where
\begin{align*}
|J|_{m,n} = \sup_{r,s \le m} \sup_{y \in \R , k \in \T} (1 + y^{2})^{n}|\partial_{y}^{r} \partial_{k}^{s} J(y, k)|.
\end{align*}

We introduce a norm $||\cdot||$ on $\mathbf{S}^{2} = \mathbf{S} \times \mathbf{S}$ defined by
\begin{align*}
||\boldsymbol{J}|| = \sum_{i = 1 , 2} \int_{\R} dp \sup_{k} |\widehat{J_{i}}(p,k)| 
\end{align*}
for $\boldsymbol{J}=(J_{1},J_{2}) \in \mathbf{S}^{2}$ and define a topology on $\mathbf{S}^{2}$ induce by the norm $||\cdot||$.

By $(\mathbf{S}^2)'$ we denote the dual space of $\mathbf{S}^2$ equipped with the weak-$*$ topology.

For two functions $f(k)$ and $g(k)$ defined on $\T$ or $\T \setminus \{0\}$, we denote by $f(k) \sim g(k)$ as $k \to 0$ if there exists a constant $C>0$ such that for all $k$ whose absolute value is small enough, $\frac{1}{C}|g(k)| \le |f(k)| \le C|g(k)|$.

\section{The Dynamics}

We consider the one-dimensional infinite chain of coupled charged harmonic oscillators in two-dimensional space with weak continuous noise. Since the dynamics involves infinite number of particles, we give a formal description of the deterministic dynamics in Section 3.1, a formal construction of the associated wave functions in Section 3.2 and a formal description of the stochastic perturbation in Section 3.3. They are rigorous when we consider a finite chain. Then we present a proper definition of the dynamics in Section 3.4. In Section 3.5 we introduce the Wigner distribution associated to our wave functions.

\subsection{Deterministic dynamics.}

We consider a one-dimensional chain of oscillators in a magnetic field. Our deterministic dynamics $ ( \mathbf{v}_x(t), \mathbf{q}_x(t) ) \in \R^{2} \times \R^{2}$ is formally given as follows:
\begin{align}\label{eq:dynamics} 
\begin{cases}
\frac{d}{dt} q_x^i & =v_x^i \\
\frac{d}{dt} v_x^i & =[\Delta q^i]_x +\delta_{i,1}Bv^2_x-\delta_{i,2}B v^1_x 
\end{cases}
\end{align}
for $x \in \Z, i =1,2$ where $B \in \R \setminus \{0\}$ is the strength of the magnetic field. 

The total energy $E$ of the system is formally given by
\begin{align*}
E = \sum_{i=1,2} \sum_{x \in \Z } \left( \frac{|v_x^i|^{2}}{2} + \frac{|q_x^i - q_{x+1}^i|^{2}}{2} \right).
\end{align*}

We introduce operators $A$ and $G$ as follows:
\begin{align*}
A  & = \sum_{i=1,2} \sum_{x \in \Z}( v_{x}^i  \partial_{q_{x}^i} + [\Delta q^{i}]_x \partial_{v_{x}^i}) , \\
G  & = \sum_{x \in \Z} \big( v_{x}^2  \partial_{v_{x}^1} - v_{x}^1  \partial_{v_{x}^2} \big). 
\end{align*}
Then our deterministic dynamics formally satisfies $\frac{d}{dt}f( \mathbf{v}, \mathbf{q})=(A+BG)f( \mathbf{v}, \mathbf{q})$ for any smooth cylinder function $f$, that is, $f$ depends on the configuration $( \mathbf{v}, \mathbf{q})$ only through a finite set of coordinates.

Let $\a: \Z \to \R$ be a function that $\a(0) = 2 $, $\a(1)= \a(-1) = -1$ and $\a(x) = 0, |x| \ge 2$. Using this function, the total energy $E$ and the operator $A$ are also written as follows:
\begin{align*}
E  &= \sum_{i=1,2} \left(\sum_{x \in \Z } \frac{|v_{x}^i|^{2}}{2} + \sum_{ x , x' \in \Z } \frac{\a(x-x')}{2} q_{x}^iq_{x'}^i \right), \\
A  & = \sum_{i=1,2} \left(\sum_{x \in \Z} v_{x}^i  \partial_{q_{x}^i} - \sum_{x , x' \in \Z} \a(x-x')q_{x'}^i \partial_{v_{x}^i}\right).
\end{align*}

\begin{remark}\label{generala}
Suppose that $\a_* : \Z \to \R$ is a function satisfying the following conditions $(a.1) - (a.4)$.

$(a.1) ~ \a_* (x) \neq 0 $ for some  $x \in \Z. $

$(a.2) ~ \a_* (x) = \a_* (-x) $ for all $ x \in \Z.$

$(a.3) ~ $ There exist some positive constants $C_{1} , C_{2}$ such that $|\a_* (x)| \le C_{1}e^{-C_{2}|x|} $ for all $x \in \Z$.

$(a.4) ~ \widehat{\a_* }(k) >0 $ for all $k \neq 0$ , $\widehat{\a_* }(0) = 0 , \widehat{\a_* }''(0) > 0$.

We can consider the dynamics associated to $\a_*$, or precisely that given by $A_* + BG$ where
\begin{align*}
A_* = \sum_{i=1,2}  \left( \sum_{x \in \Z} v_{x}^i  \partial_{q_{x}^i} - \sum_{x , x' \in \Z} \a_* (x-x')q_{x'}^i\partial_{v_{x}^i}\right).
\end{align*}
Then, Theorem 1, 2, and 3 are generalized to this dynamics (with stochastic perturbation) by replacing $\a$ with $\a_* $. The generalization from $\a$ to $\a_* $ is straightforward, so we omit the proof.
\end{remark}

\subsection{Wave functions}

To define our dynamics rigorously and then introduce the Wigner distribution, we consider the Fourier transform of the configuration $( \mathbf{v}, \mathbf{q})$. 
From the formal description of the dynamics \eqref{eq:dynamics}, the time evolution of the deterministic process $( \widehat{\mathbf{v}}(k,t), \widehat{\mathbf{q}}(k,t) )$ should be given by
\begin{align}\label{eq:dynamicsinF}
\partial_{t} ~ \begin{pmatrix} \widehat{q^1}(k,t) \\ \widehat{q^2}(k,t) \\ \widehat{v^1}(k,t)  \\ \widehat{v^2}(k,t) \end{pmatrix} = M(k) ~ \begin{pmatrix} \widehat{q^1}(k,t) \\ \widehat{q^2}(k,t) \\ \widehat{v^1}(k,t)  \\ \widehat{v^2}(k,t) \end{pmatrix} , \\
M(k) = \begin{pmatrix} 0 & 0 & 1 & 0 \\ 0 & 0 & 0 & 1 \\ -\widehat{\a}(k) & 0 & 0 & B \\ 0 & -\widehat{\a}(k) & - B & 0 \end{pmatrix} \notag,
\end{align}
for each $k \in \T$ where $\widehat{\a}(k) = 2 - 2\cos{2 \pi k}$. Note that the dynamics \eqref{eq:dynamicsinF} is well-defined for any initial condition $( \widehat{\mathbf{v}}(k,0), \widehat{\mathbf{q}}(k,0) )$ for each $k \in \T$.

We denote the eigenvalues of the matrix $M(k)$ by $\{ \pm \sqrt{-1} {\omega}_i(k) , i = 1,2 \}$, which are explicitly given as
\begin{align*}
{\omega}_1(k)  &= \sqrt{\widehat{\a}(k) + \frac{B^{2}}{4}} + \frac{B}{2} , \\
{\omega}_2(k)  &= \sqrt{\widehat{\a}(k) + \frac{B^{2}}{4}} - \frac{B}{2} .
\end{align*}
Note that ${\omega}_i(k), \omega'_i(k), ~ i=1,2$ are bounded in $k \in \T$ and $\omega'_1 = \omega'_2$. Denote by $\omega'(k)$ the common value of $\omega'_i(k)$. We introduce the corresponding wave functions $\{ \widehat{\psi_{i}}(k,t) ; i = 1,2 \}$ given by
\begin{align}\label{defofpsi1}
\widehat{\psi_{1}}(k,t) &= \theta_{1}(k)(\widehat{v^1}(k,t) - \sqrt{-1}{\omega}_2(k)\widehat{q^1}(k,t) + \sqrt{-1}\widehat{v^2}(k,t) + {\omega}_2(k)\widehat{q^2}(k,t)) ,\notag \\
\widehat{\psi_{2}}(k,t) &= \theta_{2}(k)(\widehat{v^1}(k,t) - \sqrt{-1}{\omega}_1(k)\widehat{q^1}(k,t) - \sqrt{-1}\widehat{v^2}(k,t) - {\omega}_1(k)\widehat{q^2}(k,t))
\end{align}
with
\begin{align*}
\theta_{i}(k) &= \sqrt{\frac{{\omega}_i(k)}{{\omega}_1(k)+{\omega}_2(k)}} , ~ i=1,2. 
\end{align*}
$\widehat{\psi_{i}}(k)$ is the eigenfunction associated to the eigenvalue $- \sqrt{-1} {\omega}_i(k)$ :
\begin{align*}
\partial_t \widehat{\psi_{i}}(k) = - \sqrt{-1} {\omega}_i(k) \widehat{\psi_{i}}(k) , ~ i=1,2.
\end{align*}
We normalize $\widehat{\psi}$ by multiplying $\theta_{i}$ so that the total energy $E$ is given by the integral of the $L^2$ norm of the wave functions as
\begin{align*}
E  & = \frac{1}{2} \int_{\T} dk ~ \left(|\widehat{v^{1}}(k)|^{2}+|\widehat{v^{2}}(k)|^{2}+\widehat{\a}(k)(|\widehat{q^{1}}(k)|^{2}+|\widehat{q^{2}}(k)|^{2}) \right) \\
&= \frac{1}{2}\int_{\T} dk ~ \left(|\widehat{\psi_{1}}(k)|^{2} + |\widehat{\psi_{2}}(k)|^{2}\right).
\end{align*}
By a direct computation we have
\begin{align}\label{repbypsi}
\widehat{v_{1}}(k) &= \frac{\theta_{1}(k)}{2}(\widehat{\psi_{1}}(k) + \widehat{\psi_{1}}(-k)^{*}) + \frac{\theta_{2}(k)}{2}(\widehat{\psi_{2}}(k) + \widehat{\psi_{2}}(-k)^{*}) , \notag \\
\widehat{v_{2}}(k) &= - \frac{\sqrt{-1} \theta_{1}(k)}{2}(\widehat{\psi_{1}}(k) - \widehat{\psi_{1}}(-k)^{*}) + \frac{\sqrt{-1} \theta_{2}(k)}{2}(\widehat{\psi_{2}}(k) - \widehat{\psi_{2}}(-k)^{*}) , \notag \\
\widehat{q_{1}}(k) &= \frac{\sqrt{-1} \theta_{1}(k)}{2{\omega}_1(k)}(\widehat{\psi_{1}}(k) - \widehat{\psi_{1}}(-k)^{*}) + \frac{\sqrt{-1} \theta_{2}(k)}{2{\omega}_2(k)}(\widehat{\psi_{2}}(k) - \widehat{\psi_{2}}(-k)^{*}) , \notag \\
\widehat{q_{2}}(k) &= \frac{\theta_{1}(k)}{2{\omega}_1(k)}(\widehat{\psi_{1}}(k) + \widehat{\psi_{1}}(-k)^{*}) - \frac{\theta_{2}(k)}{2{\omega}_2(k)}(\widehat{\psi_{2}}(k) + \widehat{\psi_{2}}(-k)^{*}). 
\end{align}


\subsection{Stochastic perturbation.}

We consider a local stochastic perturbation of the dynamics \eqref{eq:dynamics} which conserves the total energy. We introduce an operator $S$ as follows:
\begin{align*}
S  &= \frac{1}{2} \sum_{x \in \Z} (Y_{x,x+1})^{2} = \frac{1}{4} \sum_{x \in \Z} \sum_{z \in \Z;|x-z| = 1} (Y_{x,z})^{2}, \\
Y_{x , z}  &= (v^2_z - v_x^2) (\partial_{v^1_z}-\partial_{v_{x}^1}) -(v^{1}_z - v_{x}^1)(\partial_{v^2_z}-\partial_{v_{x}^2}). 
\end{align*}
We consider a Markov process $( \mathbf{v}_x(t) , \mathbf{q}_x(t) )$ generated by $L := A + BG +\epsilon \gamma S$. $\gamma > 0$ is the strength of the stochastic noise and $0< \epsilon <1$ is a scale parameter. The dynamics can be also given by the stochastic differential equation
\begin{align}\label{formalz}
\begin{cases}
d q_x^i & =v_x^i dt \\
d v_x^i & = ( - [\a * q^{i}]_x +\delta_{i,1}Bv^{2}_x-\delta_{i,2}B v^{1}_x + \epsilon \gamma [\Delta v^{i}]_x ) dt \\
& ~ + \sqrt{\epsilon \gamma} \sum_{z; |z-x| = 1} (Y_{x , z} v_x^i) dw_{x,z} ,
\end{cases}
\end{align}
for $x \in \Z$, $i=1,2$ where $\{ w_{x,z}(t) = w_{z,x}(t) ; x,z \in \Z, |z-x| = 1 \}$ are independent standard Wiener processes on $\R$. Note that $L$ formally conserves the total energy and the total pseudomomentum $\sum_{x} v_{x}^1 - Bq^2_x, \sum_{x} v_{x}^2 + Bq^1_x$. For more details about the conserved quantities, see \cite{SS}.

\begin{remark}
This specific choice of noise is not important. Our proof is also applicable for the
velocity exchange noise used in \cite{SS} and yields the same scaling limits. For the construction of this jump-type process, we can follow the  argument in Chapter 5 of \cite{FFL}.
\end{remark}

\subsection{Rigorous definition of the dynamics}

In this subsection, we define the dynamics rigorously. First, we calculate the time evolution of the wave functions $\widehat{\psi_{i}}(k,t)$ obtained from the formal description \eqref{formalz}: 
\begin{align}\label{formalk}
d \widehat{q^i}(k,t) & = \widehat{v^i}(k,t) dt ~ , i=1,2 \notag ,\\
d \widehat{v^1}(k,t) &= (- \widehat{\a}(k) \widehat{q^1}(k,t) + B\widehat{v^2}(k,t) + \epsilon \gamma \widehat{\b}(k) \widehat{v^1}(k,t) ) dt \notag \\
& ~ - \sqrt{\epsilon \gamma} \int_{\T} r(k,k') \widehat{v^{2}} (k-k',t) W(dk',dt) \notag ,\\
d \widehat{v^2}(k,t) &= (- \widehat{\a}(k) \widehat{q^2}(k,t) - B\widehat{v^1}(k,t) + \epsilon \gamma \widehat{\b}(k) \widehat{v^2}(k,t) ) dt \notag \\
& ~ + \sqrt{\epsilon \gamma} \int_{\T} r(k,k') \widehat{v^{1}} (k-k',t) W(dk',dt) ,
\end{align}
where 
\begin{align*}
\widehat{\b}(k) &= 2 \cos{2 \pi k} - 2 ,\\
r(k,k') &= (e^{-2 \pi \sqrt{-1} k'} - e^{-2 \pi \sqrt{-1} k})(e^{2 \pi \sqrt{-1} k} - 1) ,\\
W(k,t) &= \sum_{x \in \Z} w_{x,x+1}(t) e^{-2 \pi \sqrt{-1} k x}.
\end{align*}
The term with $\widehat{\b}(k)$ comes from the stochastic perturbation. In our case $\widehat{\a}(k) = - \widehat{\b}(k)$, but in general (cf. Remark \ref{generala}) there is no such relation between $\widehat{\a}$ and $\widehat{\b}$, and so we keep $\widehat{\a}$ and $\widehat{\b}$ for the generalization. $W$ is called a cylindrical Wiener process on $\mathbb{L}^{2}(\T)$. A precise derivation of \eqref{formalk} from \eqref{formalz} is given in Appendix \ref{derivation}. Combining $(\ref{defofpsi1})$ and $(\ref{formalk})$ we have
\begin{align}\label{defofpsi}
& d \widehat{\psi_{1}}(k,t) = (- \sqrt{-1}{\omega}_1(k) \widehat{\psi_{1}}(k,t) + \epsilon \gamma \theta_{1}(k) \b(k) ( \theta_{1}(k) \widehat{\psi_{1}}(k,t) + \theta_{2}(k) \widehat{\psi_{2}}(-k,t)^{*}) )dt \notag \\
& ~ + \sqrt{-1} \theta_{1}(k) \sqrt{\epsilon \gamma} \int_{\T} r(k,k') (\theta_{1}(k-k') \widehat{\psi_{1}}(k-k',t) + \theta_{2}(k-k') \widehat{\psi_{2}}(k'-k,t)^{*} ) W(dk',dt) , \notag \\
& d \widehat{\psi_{2}}(k,t) = (- \sqrt{-1}{\omega}_2(k) \widehat{\psi_{2}}(k,t) + \epsilon \gamma \theta_{2}(k) \b(k) ( \theta_{1}(k) \widehat{\psi_{1}}(-k,t)^{*} + \theta_{2}(k) \widehat{\psi_{2}}(k,t)) ) dt \notag \\
& ~ - \sqrt{-1} \theta_{2}(k) \sqrt{\epsilon \gamma} \int_{\T} r(k,k') (\theta_{1}(k-k') \widehat{\psi_{1}}(k'-k,t)^{*} + \theta_{2}(k-k') \widehat{\psi_{2}}(k-k',t) ) W(dk',dt) . 
\end{align}


Now we define a stochastic process $\{ \widehat{\boldsymbol{\psi}}(\cdot,t) \in (\mathbb{L}^{2}(\T))^{2} ; t \ge 0 \} $ as the unique solution of \eqref{defofpsi}. We can show the existence of the solution by using a classical technique, called a fixed-point theorem. For the sketch of the proof, see Appendix \ref{existence}. Once we define the dynamics $ \widehat{\boldsymbol{\psi}}(\cdot,t) \in (\mathbb{L}^{2}(\T))^{2} $, then we can also define $\widehat{\mathbf{v}}(k,t)$ by \eqref{repbypsi} and then define a stochastic process $\{ \mathbf{v}_x(t), \boldsymbol{\psi}(x,t) ; x \in \Z , t \ge 0  \}$ by 
\begin{align*}
&v_x^i(t) = \int_{\T} dk ~ e^{2\pi \sqrt{-1} k x} \widehat{v_{i}}(k,t) , \\
&\psi_i(x,t) = \int_{\T} dk ~ e^{2\pi \sqrt{-1} k x} \widehat{\psi_{i}}(k,t)
\end{align*}
for $x \in \Z, i= 1,2$. On the other hand, $\widehat{\mathbf{q}}(\cdot,t)$ is not necessarily well-defined as an element of $(\mathbb{L}^{2}(\T))^2$ because ${\omega}_2(k) \sim k^{2}$ as $k \to 0$ if $B>0$ and ${\omega}_1(k) \sim k^{2}$ as $k \to 0$ if $B<0$. Hence, $\mathbf{q}_x(t)$ are also not necessarily well-defined. Hereafter we do not use the variables $\mathbf{q}_x$.

\subsection{Wigner distribution.}

Let $Q_{\epsilon}$ be a probability measure on $(\mathbb{L}^{2}(\T))^{2} $ which satisfies the following condition:
\begin{align}\label{cond.initial}
K_{0}  = \sup_{0 < \epsilon < 1} \sum_{i=1,2} \epsilon \int_{\T} dk ~ E_{Q_{\epsilon}}[ |\widehat{\psi_{1}}(k)|^{2} + |\widehat{\psi_{2}}(k)|^{2}]  ~ < ~ \infty. 
\end{align}
Denote by $\mathbb{E}_{\epsilon}$ the expectation with respect to the distribution of $\{\widehat{\psi_{i}}(\cdot,t)\}_{t \ge 0}$ which starts from $Q_{\epsilon}$. In Appendix \ref{conservation}, we show that 
\[
\sum_{i=1,2} \mathbb{E}_{\epsilon}[ ||\widehat{\psi_{i}}(\cdot,t)||_{\mathbb{L}_{2}}^{2} ] = \sum_{i=1,2} \mathbb{E}_{\epsilon} [ ||\widehat{\psi_{i}}(\cdot,0)||_{\mathbb{L}_{2}}^{2} ]
\]
for any $t \ge 0$. In particular, under the condition $(\ref{cond.initial})$ 
\begin{align}\label{ebound}
\sup_{0 < \epsilon < 1} \sum_{i=1,2} \epsilon \int_{\T} dk ~ \mathbb{E}_{\epsilon}[ |\widehat{\psi_{1}}(k,t)|^{2} + |\widehat{\psi_{2}}(k,t)|^{2} ] ~ = K_{0} ~ < ~ \infty
\end{align}
for any time $t \ge 0$.

For the wave function $\boldsymbol{\psi}$, we introduce the averaged Wigner function as in Section 3 of \cite{BOS}. We denote the Wigner distribution on the time scale $\epsilon^{-1}t$ by $\Omega^{\epsilon}(t)$ with $\epsilon$ the small semiclassical parameter. Namely, we define $\Omega^{\epsilon}(t) \in (\mathbf{S}^2)'$ by
\begin{align*}
&<\Omega^{\epsilon}(t),\boldsymbol{J}> ~ = \sum_{i=1,2} <\Omega_{i}^{\epsilon}(t),J_{i}> 
\end{align*}
for $\boldsymbol{J}=(J_{1},J_{2}) \in \mathbf{S}^{2}$ with
\begin{align}\label{defofWigner}
&<\Omega_{i}^{\epsilon}(t),J>  \notag \\
&= \frac{\epsilon}{2} \sum_{x,x' \in \Z} \mathbb{E}_{\epsilon}[\psi_{i}(x',\frac{t}{\epsilon})^{*} \psi_{i}(x,\frac{t}{\epsilon}) ] \int_{\T} dk ~ e^{2\pi \sqrt{-1} (x'-x) k} J(\frac{\epsilon}{2} (x+x'),k)^{*} \notag \\
&= \frac{\epsilon}{2} \int_{\R} dp \int_{\T} dk ~ \mathbb{E}_{\epsilon} [\widehat{\psi_{i}}(k-\frac{\epsilon p}{2},\frac{t}{\epsilon})^{*} ~  \widehat{\psi_{i}}(k+\frac{\epsilon p}{2},\frac{t}{\epsilon}) ] \widehat{J}(p,k)^{*} 
\end{align}
for $J \in \mathbf{S}$. 
By the Cauchy-Schwarz inequality and $(\ref{ebound})$, 
\begin{align}\label{lbb}
\sup_{0 < \epsilon < 1} \sup_{t \ge 0} |<\Omega^{\epsilon}(t),\boldsymbol{J}>| ~ \le ~ \frac{1}{2} K_{0}||\boldsymbol{J}||
\end{align}
under the condition $(\ref{cond.initial})$. 
\begin{remark}
As discussed in \cite{BOS}, $\Omega^{\epsilon}(\cdot)$ is well-defined on a wider class of test functions than $\mathbf{S}^2$. In particular we can take $\boldsymbol{J}(y,k) =(J(k),J(k))$ with a bounded function $J(k)$ on $\T$, and then we have
\begin{align*}
<\Omega^{\epsilon}(t),\boldsymbol{J}> = \frac{\epsilon}{2} \int_{\T} dk ~ \sum_{i=1,2}\mathbb{E}_{\epsilon} [|\widehat{\psi_{i}}(k,\frac{t}{\epsilon})|^{2} ] J(k).
\end{align*}
From this representation one can see that $\Omega^{\epsilon}(\cdot)$ is the distribution of the spectral density of the energy. Also if we take $\boldsymbol{J}(y,k) =(J(y),J(y))$ with a rapidly decreasing function $J(y)$ on $\R$ as a test function, then we have
\begin{align*}
<\Omega^{\epsilon}(t),\boldsymbol{J}> = \frac{\epsilon}{2} \sum_{x \in \Z} \sum_{i=1,2}\mathbb{E}_{\epsilon} [ |\psi_{i}(x,\frac{t}{\epsilon})|^{2} ] J(\epsilon x).
\end{align*}
This is the integral of $J$ with respect to the averaged empirical measure of $\frac{1}{2} \sum_{i=1,2}|\psi_{i}(x,\frac{t}{\epsilon})|^2$. Namely, $ \Omega^{\epsilon}(t)$ is a rescaled microscopic local spectral density. 
\end{remark}


\section{Main results}
As mentioned in the Introduction, the main purpose of the present paper is to understand the nature of the superdiffusion for the coupled charged harmonic chain of oscillators in a magnetic field with noise defined in the last section, and we apply the two-step scaling limits. In Subsection 4.1, following the idea of \cite{BOS}, we claim that in the weak noise limit the local density of energy is governed by a phonon linear Boltzmann equation. In Subsection 4.2, we consider a properly scaled solution of the Boltzmann equation and state that it converges to the solution of the fractional diffusion equation \eqref{fractional} with $s=\frac{5}{3}$, which is our main result.

\subsection{Boltzmann equation}

In this subsection we state the limiting behavior of the Wigner distribution. 
\begin{theorem}\label{thm:main}
Suppose the condition \eqref{cond.initial} holds. If $\Omega^{\epsilon}(0)$ converges to $\Omega_0$ in $(\mathbf{S}^2)'$ as $\epsilon \to 0$, then for all $t \ge 0$, $\Omega^{\epsilon}(t)$ converges to a vector-valued finite positive measure $\boldsymbol{\mu}(t)=(\mu_{1}(t),\mu_{2}(t))$ in $(\mathbf{S}^2)'$ as $\epsilon \to 0$, which is the unique solution of the following Boltzmann equation 
\begin{align}\label{mboltzmann}
\begin{cases}
\partial_{t} \int d \boldsymbol{\mu}(t) \cdot \boldsymbol{J} = \frac{1}{2\pi} \int d\boldsymbol{\mu}(t) \cdot \omega' \partial_{y} \boldsymbol{J} + \gamma \int d\boldsymbol{\mu}(t)\cdot C\boldsymbol{J} \\
\int d \boldsymbol{\mu}(0) \cdot \boldsymbol{J} = < \Omega_0  ,  \boldsymbol{J}>,
\end{cases}
\end{align}
where
\begin{align*}
\int d \boldsymbol{\mu} \cdot \boldsymbol{J} &= \sum_{i=1,2} \int_{\R \times \T} \mu_{i} (dy,dk) ~ J_{i}(y,k)^* \quad \text{for} \quad \boldsymbol{\mu}=(\mu_1,\mu_2),\\
(C\boldsymbol{J})_{i}(x,k) &= \sum_{j = 1,2} \int_{\T} dk' \theta_{i}(k)^{2} R(k,k') \theta_{j}(k')^{2} (J_{j}(x,k')-J_{i}(x,k))  
\end{align*}
for $\boldsymbol{J}=(J_1,J_2)\in \mathbf{S}^{2}$ with $R(k,k') = 16\sin^{2}{\pi k}\sin^{2}{\pi k'}$.

\end{theorem}

\begin{remark}
In the case $B = 0$, if we assume an additional assumption
\begin{align*}
\lim_{\rho \to 0} \limsup_{\epsilon \to 0} \frac{\epsilon}{2} \sum_{i=1}^2\int_{|k| < \rho} dk ~ E_{Q_{\epsilon}}[|\widehat{\psi_{i}}(k)|^{2}] = 0 
\end{align*}
on the initial measure $Q_{\epsilon}$, the same statement of Theorem \ref{thm:main} holds. For this case, the proof is essentially given in \cite{BOS}.

\end{remark}

\begin{remark}
Suppose that the solution of (\ref{mboltzmann}) has the density $u(y,k,i,t)$ for all $t \in [0,T]$ , that is, 
\begin{align*}
\mu_{i}(t)(dy,dk)  &= u(y,k,i,t) dy dk ~ , ~ i = 1,2 , \\
\mu_{i}(0)(dy,dk)  &= u_{0}(y,k,i) dy dk ~ , ~ i = 1,2 .
\end{align*}
Then $u(y,k,i,t)$ is a weak solution of the linear Boltzmann equation
\begin{align}\label{boltzmann}
\begin{cases}
\partial_{t} u(y,k,i,t) + \frac{1}{2\pi} \omega'(k) \partial_{y} u(y,k,i,t) = \gamma \mathcal{L} u(y,k,i,t) \\
u(y,k,i,0) = u_{0}(y,k,i) , \\
\end{cases}
\end{align}
where
\begin{align*}
\mathcal{L}u(y,k,i,t) = \sum_{j= 1,2} \int_{\mathbf{T}} dk' \theta_{i}(k)^2 R(k,k') \theta_{j}(k')^2 (u(y,k',j,t)-u(y,k,i,t) ) .
\end{align*}
\end{remark}

We prove Theorem \ref{thm:main} in Section \ref{sec:proof1}. The strategy of our proof is as follows:
First we derive a microscopic evolution equation of $\Omega^{\epsilon}$, which is not closed in terms of $\Omega^{\epsilon}$. Then, with this expression of the time evolution, we show that for any fixed $T>0$, $\{ \Omega^{\epsilon}(t) , 0 \le t \le T \}_{0 < \epsilon < 1}$ is sequentially compact in $C([0,T];(\mathbf{S}^2)')$ in a certain weak-$*$ sense. See its precise meaning in Section \ref{sec:proof1}. We verify that any limit of a convergent subsequence is extended to a vector-valued finite positive measure in Appendix \ref{positivity}. The uniqueness of the bounded solution of $(\ref{mboltzmann})$ in the class of vector-valued finite positive measures is shown in Appendix \ref{uniqueness}. Finally we show that any limit of a convergent subsequence satisfies $(\ref{mboltzmann})$, which is a closed equation in terms of $\boldsymbol{\mu}$. Summarizing the above we can show that $( \Omega^{\epsilon}(\cdot) )_{\epsilon}$ is convergent and the limit satisfies $(\ref{mboltzmann})$.

\subsection{Derivation of the $\frac{5}{6}$ fractional diffusion equation.}

In this subsection we study a macroscopic behavior of a solution of properly scaled Boltzmann equation \eqref{boltzmann}.
Consider a spatially scaled linear Boltzmann equation with a scaling parameter $N$ as
\begin{align}\label{boltzmannN}
\begin{cases}
\partial_{t} u(y,k,i,t) + \frac{1}{N^{3/5}}\frac{1}{2\pi} \omega'(k) \partial_{y} u(y,k,i,t) = \gamma \mathcal{L} u(y,k,i,t) \\
u(y,k,i,0) = u_{0}(y,k,i) , \\
\end{cases}
\end{align}
and denote its solution by $u_N$.

\begin{remark}
For any given $u_{0}(y,k,i) \in C_{0}^{\infty}(\R \times \T) , ~ i=1,2$, a solution of \eqref{boltzmann} is constructed explicitly using a Markov process associated to the Boltzmann equation in the next section. The uniqueness of solutions in a certain class follows from that of (\ref{mboltzmann}). The argument also applies to \eqref{boltzmannN} and so the existence and uniqueness of $u_N$ follows.
\end{remark}

\begin{theorem}\label{thm:main2}
Suppose $u_{0}(y,k,i) \in C_{0}^{\infty}(\R \times \T) , ~ i=1,2$. Define the initial local density of energy at $y \in \R$ as $\bar{u}_0(y) =\sum_{i=1,2} \int_{\T \times \{ 1,2 \} } dk ~ u_{0}(y,k,i)$. Then, for all $y \in \R, ~ t \ge 0$, 
\begin{align*}
\lim_{N \to \infty} \sum_{i=1,2} \int_{\T} dk ~ |u_{N}(y,k,i,Nt) - \frac{1}{2}\bar{u}(y,t)|^2 = 0 ,
\end{align*}
where $\bar{u}$ is a solution of 
\begin{align}\label{boltzmann2}
\begin{cases}
\partial_{t} \bar{u}(y,t) = - D (-\Delta_{y})^{\frac{5}{6}} \bar{u}(y,t) \\
\bar{u}(y,0) = \bar{u}_0(y) 
\end{cases}
\end{align}
and $D = D(B,\gamma,\a) $ is a positive constant such that
\begin{align*}
D = C |B|^{-\frac{1}{3}} \gamma^{-\frac{2}{3}} \widehat{\a}''(0) 
\end{align*}
with a universal constant $C$.
In particular,
\begin{align*}
\lim_{N \to \infty} ~ | \sum_{i=1,2} \int_{\T} dk \  u_{N}(y,k,i,Nt) - \bar{u}(y,t)|^2 = 0.
\end{align*}

\end{theorem}

\begin{remark}
In the case $B = 0$, if we denote by $u_N(y,k,i,t)$ the solution of a scaled linear Boltzmann equation
\begin{align*}
\begin{cases}
\partial_{t} u(y,k,i,t) + \frac{1}{N^{2/3}} \frac{1}{2\pi} \omega'(k) \partial_{y} u(y,k,i,t) = \gamma \mathcal{L} u(y,k,i,t) \\
u(y,k,i,0) = u_{0}(y,k,i) , 
\end{cases}
\end{align*}
then for all $y \in \R,~ t \ge 0$, 
\begin{align*}
\lim_{N \to \infty} \sum_{i=1,2} \int_{\T} dk ~ |u_{N}(y,k,i,Nt) - \frac{1}{2}\bar{u}(y,t)|^2 = 0 
\end{align*}
where $\bar{u}$ is the solution of 
\begin{align*}
\begin{cases}
\partial_{t} \bar{u}(y,t) = - D' (-\Delta_{y})^{\frac{3}{4}} \bar{u}(y,t) \\
\bar{u}(y,0) = \bar{u}_0(y) .
\end{cases}
\end{align*}
and $D' = D'(\gamma,\a) $ is a positive constant such that
\begin{align*}
D' = C'\gamma^{-\frac{1}{2}} (\widehat{\a}''(0))^{\frac{3}{4}}
\end{align*}
with a universal constant $C'$.
The result is essentially proved in \cite{KJO}.
\end{remark}

For the proof, we follow the strategy of \cite{KJO}. Namely, we consider a long-time asymptotic behavior of a Markov process associated to the Boltzmann equation \eqref{boltzmann} and then use a functional limit theorem for a continuous time random walk. To apply a general theorem in \cite{KJO}, we need to check several conditions. This is the main subject of the next section, where we conclude all the required conditions are satisfied and then Theorem \ref{thm:main3} on the asymptotic behavior of a Markov process is obtained. We apply it to the study of the limit of $u_N$ and prove Theorem \ref{thm:main2} in Section \ref{sec:proof2}.

\section{Markov process associated to the Boltzmann equation}\label{markov}

In this section we construct a solution of $(\ref{boltzmann})$ probabilistically. We will see that there exists a Markov process associated to $(\ref{boltzmann})$ and study its long-time asymptotic behavior. 

Let $\{ ( K_{n} , I_{n} ) ; n \in \Z_{\ge 0} \}$ be a Markov chain on $\T \times \{ 1 , 2 \}$ whose transition probability is given by 
\begin{align*}
P(k,i,dk',j) = t(k,i) \gamma \theta_{i}(k)^2 R(k,k') \theta_{j}(k')^2 dk',
\end{align*}
where 
\begin{align*}
t(k,i)  &= [ \gamma \theta_{i}(k)^2 R(k) ]^{-1} , \quad R(k) = \int_{\T} dk' R(k,k').
\end{align*}
Since $R(k,k')$ is a product of functions of $k$ and $k'$, we have 
\begin{align*}
P(k,i,dk',j) = \pi(dk',j)
\end{align*}
where $\pi(dk,di)$ is a reversible measure for this Markov chain given as
\begin{align*}
\pi(dk,di) = \sum_{j=1,2} \frac{t(k,j)^{-1}}{\gamma \overline{R}} dk \delta_{ \{ j \} }(di) ,  \quad \overline{R} = \int_{\T} dk ~ R(k).
\end{align*}
In particular, $\{ ( K_{n} , I_{n} ) ; n \ge 1 \}$ is an i.i.d. sequence of random variables on $\T \times \{ 1 , 2 \}$ with distribution $\pi$.

Now we construct a continuous time random walk generated by $\mathcal{L}$. Let $\{ \tau_{n} , n \ge 1 \}$ be an i.i.d. sequence of random variables such that $\tau_{1}$ is exponentially distributed with intensity 1 and  $\{ ( K_{n} , I_{n} ) ; n \in \Z_{\ge 0} \}$ and $\{ \tau_{n} , n \ge 1 \}$ are independent. Set $t_{n} := \sum_{m = 1}^{n} t(K_{m-1},I_{m-1}) \tau_{m} , n \ge 1 , ~ t_{0} = 0 $ and define a stochastic process $( K(t) , I(t) ) $ as $K(t) = K_{n} , I(t) = I_{n} $ if $t \in [t_{n},t_{n+1})$. Then, by the construction $\{ ( K(t) , I(t) ) \}_{t \ge 0}$ is a continuous time random walk generated by $\mathcal{L}$. With this process we can construct an explicit solution of the equation (\ref{boltzmann}) by
\begin{align*}
u(y,k,i,t)  &= \mathbb{E}_{(k,i)}[u_{0}(Z(t),K(t),I(t))] ,
\intertext{where}
Z(t)  &= y + \frac{1}{2\pi} \int_{0}^{t} ds ~ \omega'(K(s)) ,
\end{align*}
and $K(0) = k , I(0) = i$. For this process, we have the following result.
\begin{theorem}\label{thm:main3}
Suppose $(K(0), I(0))=(k,i)$ for some $k \neq 0$ and $i=1$ or $2$. Then as $N \to \infty$, the finite-dimensional distribution of scaled processes \nolinebreak{$\{ N^{-\frac{3}{5}} Z(Nt)\}_{t \ge 0 }$} converge weakly to a  L\'{e}vy  process generated by $- D (-\Delta_{y})^{\frac{5}{6}}$, where $D = D(B,\gamma,\a) $ is a positive constant such that
\begin{align*}
D = C |B|^{-\frac{1}{3}} \gamma^{-\frac{2}{3}} \widehat{\a}''(0),
\end{align*}
and C is a positive constant which does not depend on B, $\gamma$, $\a$. 
\end{theorem}
\begin{remark}
In the case of $B = 0$,the finite-dimensional distribution of scaled processes \nolinebreak{$\{ N^{-\frac{2}{3}} Z(Nt) ; t \ge 0 \}$} converge weakly to a  L\'{e}vy  process generated by $ - D' (-\Delta_{y})^{\frac{3}{4}}$, where $D' = D'(\gamma,\a) $ is a positive constant such that
\begin{align*}
D' = C'\gamma^{-\frac{1}{2}} (\widehat{\a}''(0))^{\frac{3}{4}} ,
\end{align*}
and $C'$ is a positive constant which does not depend on $\gamma$, $\a$. It is essentially shown in \cite{KJO}.
\end{remark}

\subsection{Proof of Theorem \ref{thm:main3}.}\label{sec:proof3}
We apply \cite[Theorem 2.8 (i)]{KJO} to our process with $\a=\frac{5}{3}$. For this, 
it is enough to show that Conditions 2.1, 2.2, 2.3 and (2.12) of \cite{KJO} are satisfied.

First we verify that Condition 2.1 is satisfied. Define
\begin{align*}
\Psi(k,i) = \omega'(k) t(k,i) .
\end{align*}
The tail of $\Psi$ under $\pi$ is
\begin{align*}
\pi(\{ (k,i)  ; \Psi(k,i) \ge \lambda \}) & = \sum_{i = 1,2} \int_{ \{ k ; \Psi(k,i) \ge \lambda \} } dk \ \frac{\theta_{i}(k)^2R(k)}{\overline{R}} \\
&= C |B|^{- \frac{1}{3}} \gamma^{- \frac{5}{3}} \widehat{\a}''(0) \lambda^{-\frac{5}{3}}(1 + O(\lambda^{-\frac{4} {3}}) ) ,
\end{align*}
as $\lambda \to \infty$ because
\begin{align*}
&\theta_{1}(k)^2 \sim 1,  \  \textit{and }  \ \theta_{2}(k)^2 \sim \frac{\widehat{\a}''(0)k^{2}}{|B|^{2}} \textit{as} \ k \to 0 \ \textit{if} \ B>0 \\
&\theta_{1}(k)^2 \sim \frac{\widehat{\a}''(0)k^{2}}{|B|^{2}}   \  \textit{and }  \  \theta_{2}(k)^2 \sim 1  \textit{as} \ k \to 0 \ \textit{if} \ B<0 \ \\
\end{align*}
and
\begin{align*}
\omega'(k) \sim \frac{\widehat{\a}(0)'' k}{|B|},  \quad R(k) \sim k^{2}  \quad \textit{as} \ k \to 0.
\end{align*}
$C$ is a positive constant which does not depend on $B,\gamma,\a$. $\Psi$ is odd for $k$ and the density of $\pi(\cdot,i)$ with respect to the Lebesgue measure is even for k, so 
\[
\pi(\{ (k,i)  ; \Psi(k,i) \ge \lambda \}) = \pi(\{ (k,i)  ; \Psi(k,i) \le -\lambda \})
\]
and $ \int \Psi d\pi = 0$.

Next we check Condition 2.2. It is obvious that 
\begin{align*}
\sup \{ ||Pf||_{\mathbb{L}^{2}(\pi)} ; \int d\pi ~ f = 0 , ||f||_{\mathbb{L}^{2}(\pi)} = 1 \} = 0
\end{align*}
because $Pf =  \int d\pi ~ f$. 

Finally we show that Condition 2.3 and (2.12) hold. Condition 2.3 is obviously satisfied with $Q \equiv 0$ and $p \equiv 1$. Also, we have
\begin{align*}
||P\Psi||_{\mathbb{L}^{2}(\pi)}^{2} &= \sum_{i=1,2} \int_{\T} dk \left(\sum_{j=1,2} \int_{\T} dk' \Psi(k',j) \frac{t(k',j)^{-1}}{\gamma \overline{R}}\right)^{2} \frac{t(k,i)^{-1}}{\gamma \overline{R}} \\
&= \sum_{i=1,2} \int_{\T} dk \left(\sum_{j=1,2} \int_{\T} dk' \frac{\omega'(k')}{\gamma \overline{R}} \right)^{2} \frac{t(k,i)^{-1}}{\gamma \overline{R}} < \infty.
\end{align*}

Therefore, by \cite[Theorem 2.8 (i)]{KJO}, the finite-dimensional distributions of the scaled process $\{ N^{-\frac{3}{5}} Z(Nt)\}_{t \ge 0 }$ under $\mathbb{P}_{\pi}$ converge to a stable process with exponent $\frac{5}{3}$ whose characteristic function at time 1, denoted by $\phi(x)$ is
\begin{align*}
\phi(x) = \exp{(\int_{\R} d\lambda ~ e^{\sqrt{-1}\lambda x} c_{*}(\lambda) |\lambda|^{-\frac{8}{3}})},
\end{align*}
where
\begin{align*}
c_{*}(\lambda) = \frac{5C |B|^{- \frac{1}{3}} \gamma^{- \frac{5}{3}} \widehat{\a}''(0) A_{\frac{5}{3}}}{\bar{t}}
\end{align*}
for all $\lambda \neq 0$, C is a positive constant appeared in the tail estimate of $\Psi$ and 
\begin{align*}
A_{\frac{5}{3}} = \int_{0}^{\infty} dy ~ e^{-y} y^{\frac{5}{3}}, \\
\bar{t} = \int d\pi ~ t(k,i) ~ = \frac{1}{2\gamma}. 
\end{align*}

Finally we show that the finite-dimensional distributions of $\{ N^{-\frac{3}{5}} Z(Nt) ; t \ge 0 \}$ under $\mathbb{P}_{(k,i)}$ also converge to the same stable process for $k \in \T \setminus \{ 0 \} , ~ i=1,2$. For $t \ge 0$ define $n(t)$ as the nonnegative integer such that
\begin{align*}
t_{n(t)} \le t < t_{n(t) + 1}.
\end{align*}
Then we have
\begin{align*}
N^{-\frac{3}{5}} Z(Nt) = N^{-\frac{3}{5}} \sum_{n = 0}^{n(Nt)} \Psi(K_{n},I_{n}) \tau_{n+1}.
\end{align*}
If $k \neq 0$ then $N^{-\frac{3}{5}} \Psi(k,i) \tau_{1} \to 0 $ as $N \to \infty$ $\mathbb{P}_{(k,i)}$ - almost surely. Moreover, under $\mathbb{P}_{(k,i)}$, the distribution of $\{ ( K_{n} , I_{n} )\}_{n \ge 1 }$ is an i.i.d. sequence with distribution $\pi$. By Theorem 6.1 and Lemma 6.2 of \cite{KJO}, the finite-dimensional distributions of $\{ N^{-\frac{3}{5}} \sum_{n = 1}^{n(Nt)} \Psi(K_{n},I_{n}) \tau_{n+1} ; t \ge 0 \}$ under $\mathbb{P}_{(k,i)}$ converge to the stable process, so the finite-dimensional distributions of $\{ N^{-\frac{3}{5}} Z(Nt) ; t \ge 0 \}$ under $\mathbb{P}_{(k,i)}$ also converge to the same stable process if $k \neq 0$.

\section{Proof of the Theorem \ref{thm:main}.}\label{sec:proof1}

To simplify the notation, we define functions $\widehat{\Omega_{i}^{\epsilon}}_{+}(t)(p,k)$, $\widehat{\Gamma_{i}^{\epsilon}}_{+}(t)(p,k)$ on $\R \times \T$ by
\begin{align*}
\widehat{\Omega_{i}^{\epsilon}}_{+}(t)(p,k) &= \frac{\epsilon}{2} \mathbb{E}_{\epsilon} [ \widehat{\psi_{i}}(k-\frac{\epsilon p}{2},\frac{t}{\epsilon})^{*} ~  \widehat{\psi_{i}}(k+\frac{\epsilon p}{2},\frac{t}{\epsilon}) ] , \notag \\
\widehat{\Gamma_{i}^{\epsilon}}_{+}(t)(p,k) &= \frac{\epsilon}{2} \mathbb{E}_{\epsilon} [ \widehat{\psi_{i}}(-k+\frac{\epsilon p}{2},\frac{t}{\epsilon}) ~ \widehat{\psi_{i^{*}}}(k+\frac{\epsilon p}{2},\frac{t}{\epsilon}) ] 
\end{align*}
for $i=1,2$ where $i^{*} = 3-i$. We use the notation $i^*$ throughout the rest of the paper.
We also define $\widehat{\Omega_{i}^{\epsilon}}_{-}(t)(p,k)$, $\widehat{\Gamma_{i}^{\epsilon}}_{-}(t)(p,k)$ as
\begin{align*}
\widehat{\Omega_{i}^{\epsilon}}_{-}(t)(p,k) &= \widehat{\Omega_{i}^{\epsilon}}_{+}(t)(p,-k), \notag \\
\widehat{\Gamma_{i}^{\epsilon}}_{-}(t)(p,k) &= \widehat{\Gamma_{i}^{\epsilon}}_{+}(t)^{*}(-p,k).
\end{align*}
Note that for all $p \in \R$ these functions satisfy
\begin{align*}
||\widehat{\Omega_{i}^{\epsilon}}_{\iota}(t)(p,\cdot)||_{\mathbb{L}^{1}(\T)} \le \frac{1}{2} K_{0}, \quad ||\widehat{\Gamma_{i}^{\epsilon}}_{\iota}(t)(p,\cdot)||_{\mathbb{L}^{1}(\T)} \le \frac{1}{2} K_{0} 
\end{align*}
for $i=1,2 $, $\iota = +,-$ under the condition $(\ref{cond.initial})$. 
With this notation, Wigner distribution is rewritten as
\begin{align}\label{rewriteWigner}
<\Omega^{\epsilon}(t),\boldsymbol{J}> &= \sum_{i=1,2} \int_{\R} dp \int_{\T} dk ~ \widehat{\Omega_{i}^{\epsilon}}_{+}(t)(p,k) \widehat{J_{i}}(p,k)^{*}.
\end{align}

From now on we will show that the time evolution of $\Omega^{\epsilon}(\cdot)$ satisfies the following equation
\begin{align}\label{timeevo}
&\partial_{t} <\Omega^{\epsilon}(t),\boldsymbol{J}> \notag \\
&= \frac{1}{2\pi} <\Omega^{\epsilon}(t),\omega'(k) \partial_{y} \boldsymbol{J}> + \gamma<\Omega^{\epsilon}(t),C\boldsymbol{J}> \notag \\
& ~ + \gamma ( <\Gamma^{\epsilon}(t),C'\boldsymbol{J}> + <(\Gamma^{\epsilon})^{*}(t),C'\boldsymbol{J}>) + O_{\boldsymbol{J}}(\epsilon)
\end{align}
for $\boldsymbol{J} \in \mathbf{S}^{2}$ where 
\begin{align}\label{rewriteWigner2}
<\Gamma^{\epsilon}(t),\boldsymbol{J}> &= \sum_{i=1,2} \int_{\R} dp \int_{\T} dk ~  \widehat{\Gamma_{i}^{\epsilon}}_{+}(t)(p,k) \widehat{J_{i}}(p,k)^{*} ,\notag \\
<(\Gamma^{\epsilon})^{*}(t),\boldsymbol{J}> &= \sum_{i=1,2} \int_{\R} dp \int_{\T} dk ~ \widehat{\Gamma_{i}^{\epsilon}}_{-}(t)(p,k) \widehat{J_{i}}(p,k)^{*}
\end{align}
and 
\begin{align*}
(C'\boldsymbol{J})_{i}(p,k) = \int_{\T} dk' ~ \theta_{1}(k) \theta_{2}(k) R(k,k') \theta_{i^{*}}^{2}(k') J_{i^*}(p,k') + \widehat{\b}(k) \theta_{1}(k) \theta_{2}(k) J_{i}(p,k).
\end{align*}
Here, $O_{\boldsymbol{J}}(\epsilon)$ is a term which satisfies
\begin{align*}
\frac{O_{\boldsymbol{J}}(\epsilon)}{\epsilon} \le C_{\boldsymbol{J}} 
\end{align*}
for all $0 < \epsilon < 1$ with a positive constant $C_{\boldsymbol{J}}$ which depends on $\boldsymbol{J}$. 


By $(\ref{defofpsi})$ the time evolution of $\widehat{\Omega_{i}^{\epsilon}}_{+}(t)(p,k)$ is 
\begin{align}\label{evoofWigner}
&\partial_{t} \widehat{\Omega_{i}^{\epsilon}}_{+}(t)(p,k) \notag \\
&= - \frac{\sqrt{-1}}{\epsilon}({\omega}_i(k+\frac{\epsilon p}{2}) - {\omega}_i(k-\frac{\epsilon p}{2})) \widehat{\Omega_{i}^{\epsilon}}_{+}(t)(p,k) \notag \\
& ~ + \gamma (\widehat{\b}(k+\frac{\epsilon p}{2}) \theta_{i}^{2}(k+\frac{\epsilon p}{2}) + \widehat{\b}(k-\frac{\epsilon p}{2}) \theta_{i}^{2}(k-\frac{\epsilon p}{2})) \widehat{\Omega_{i}^{\epsilon}}_{+}(t)(p,k) \notag \\
& ~ + \gamma \widehat{\b}(k+\frac{\epsilon p}{2}) \theta_{i}(k+\frac{\epsilon p}{2}) \theta_{i^{*}}(k+\frac{\epsilon p}{2}) \widehat{\Gamma_{i^{*}}^{\epsilon}}_{-}(t)(p,k) \notag \\
& ~ + \gamma \widehat{\b}(k-\frac{\epsilon p}{2}) \theta_{i}(k-\frac{\epsilon p}{2}) \theta_{i^{*}}(k-\frac{\epsilon p}{2})\widehat{\Gamma_{i^{*}}^{\epsilon}}_{+}(t)(p,k) \notag \\
& + \gamma \theta_{i}(k-\frac{\epsilon p}{2}) \theta_{i}(k+\frac{\epsilon p}{2}) \int_{\T} dk' R_{\epsilon p}(k,k') \notag \\
& ~ ~ \times [\theta_{i}(k'-\frac{\epsilon p}{2}) \theta_{i}(k'+\frac{\epsilon p}{2}) \widehat{\Omega_{i}^{\epsilon}}_{+}(t)(p,k') + \theta_{i^{*}}(k'-\frac{\epsilon p}{2}) \theta_{i^{*}}(k'+\frac{\epsilon p}{2}) \widehat{\Omega_{i^{*}}^{\epsilon}}_{+}(t)(p,k')) \notag \\
& ~ ~ ~ + \theta_{i}(k'-\frac{\epsilon p}{2}) \theta_{i^{*}}(k'+\frac{\epsilon p}{2}) \widehat{\Gamma_{i^{*}}^{\epsilon}}_{-}(t)(p,k') + \theta_{i}(k'+\frac{\epsilon p}{2}) \theta_{i^{*}}(k'-\frac{\epsilon p}{2}) \widehat{\Gamma_{i^{*}}^{\epsilon}}_{+}(t)(p,k')], 
\end{align}
where 
\begin{align*}
R_{\epsilon p}(k,k') = 16 \sin{(k+\frac{\epsilon p}{2})}\sin{(k-\frac{\epsilon p}{2})}\sin{(k'+\frac{\epsilon p}{2})}\sin{(k'-\frac{\epsilon p}{2})}.
\end{align*}
For the derivation of $(\ref{evoofWigner})$, see Appendix \ref{evolution}.


Since $\widehat{\b}, \theta_{i}$ and ${\omega}_i$ are smooth on $\T$, the term $(\ref{evoofWigner})$ is rewritten as
\begin{align}\label{evoofWigner2}
&\partial_{t} \widehat{\Omega_{i}^{\epsilon}}_{+}(t)(p,k) \notag \\
&= - \sqrt{-1} p \omega'_i(k) \widehat{\Omega_{i}^{\epsilon}}_{+}(t)(p,k) + 2 \gamma \widehat{\b}(k) \theta_{i}^{2}(k) \widehat{\Omega_{i}^{\epsilon}}_{+}(t)(p,k) \notag \\
& ~ + \gamma \widehat{\b}(k) \theta_{i}(k) \theta_{i^{*}}(k) (\widehat{\Gamma_{i^{*}}^{\epsilon}}_{-}(t)(p,k) + \widehat{\Gamma_{i^{*}}^{\epsilon}}_{+}(t)(p,k)) \notag \\
& + \gamma \theta_{i}^{2}(k) \int_{\T} dk' R(k,k') [\theta_{i}^{2}(k') \widehat{\Omega_{i}^{\epsilon}}_{+}(t)(p,k') + \theta_{i^{*}}^{2}(k') \widehat{\Omega_{i^{*}}^{\epsilon}}_{+}(t)(p,k') \notag \\
& ~ ~ + \theta_{i}(k') \theta_{i^{*}}(k') \widehat{\Gamma_{i^{*}}^{\epsilon}}_{-}(t)(p,k') + \theta_{i}(k') \theta_{i^{*}}(k') \widehat{\Gamma_{i^{*}}^{\epsilon}}_{+}(t)(p,k')] + \mathcal{R}_{i}(p,k), \end{align}
where $\mathcal{R}_{i}, i=1,2$ are the remainder terms and these satisfy
\begin{align*}
||\mathcal{R}_{i}(p,\cdot)||_{\mathbb{L}^{1}(\T)} \le C(T,B,\gamma,\a,K_{0}) |p| \epsilon 
\end{align*}
for all $p \in \R$. Then for any $\boldsymbol{J} \in \mathbf{S}^{2}$,
\begin{align}\label{remainder}
\int_{\R} dp \int_{\T} dk ~ \mathcal{R}_{i}(p,k) \widehat{J_{i}}(p,k)^{*} = O_{\boldsymbol{J}}(\epsilon).
\end{align}
Combining \eqref{rewriteWigner}, \eqref{rewriteWigner2}, \eqref{evoofWigner2} and \eqref{remainder} with the relation $\int_{\T} dk' R(k,k') = -2 \widehat{\b}(k)$, we conclude that \eqref{timeevo} holds.

From \eqref{lbb} and \eqref{timeevo}, for any fixed $T>0$ and $\boldsymbol{J} \in \mathbf{S}^{2}$, $\{ < \Omega^{\epsilon}(\cdot),\boldsymbol{J}>  \}_{ 0<\epsilon <1} \subset C([0,T],\C) $ is uniformly bounded and equicontinuous. Hence, for each $\boldsymbol{J} \in \mathbf{S}^{2}$, there exists a subsequence $\epsilon_N \downarrow 0$ such that $ < \Omega^{\epsilon_N}(\cdot),\boldsymbol{J}>$ converges to a function in $C([0,T], \C)$ uniformly as $N \to \infty$. Since $ \mathbf{S}^{2}$ is separable, there is a dense countable subset $\{ \boldsymbol{J}^{(m)} ; m \in \N \}$ of $\mathbf{S}^2$ and by the diagonal argument we can find a sequence $\epsilon_N \downarrow 0$ such that $ < \Omega^{\epsilon_N}(\cdot),\boldsymbol{J}^{(m)}> $ converges for all $m \in \N$. Now, we show that for all $J \in \mathbf{S}^2$, $ < \Omega^{\epsilon_N}(\cdot),\boldsymbol{J}> $ converges uniformly to a continuous function as $N \to \infty$. Fix $\boldsymbol{J} \in \mathbf{S}^2$ and $\delta > 0$. Since $\{ \boldsymbol{J}^{(m)} ; m \in \N \}$ is dense we can take some $\boldsymbol{J}^{(l)}$ so that $||\boldsymbol{J} - \boldsymbol{J}^{(l)}|| < \delta$.  Then for any $n,m \in \N$ we have
\begin{align*}
&\sup_{t \in [0,T]} | <\Omega^{\epsilon_n}(t),\boldsymbol{J}> - <\Omega^{\epsilon_m}(t),\boldsymbol{J}> | \\
&\le \sup_{t \in [0,T]} | <\Omega^{\epsilon_n}(t),\boldsymbol{J}> - <\Omega^{\epsilon_n}(t),\boldsymbol{J}^{(l)}> | \\
& ~ + \sup_{t \in [0,T]} | <\Omega^{\epsilon_n}(t),\boldsymbol{J}^{(l)}> - <\Omega^{\epsilon_m}(t),\boldsymbol{J}^{(l)}> | \\
& ~ + \sup_{t \in [0,T]} | <\Omega^{\epsilon_m)}(t),\boldsymbol{J}^{(l)}> - <\Omega^{\epsilon_m}(t),\boldsymbol{J}> | \\
&\le K_{0} \delta +  \sup_{t \in [0,T]} | <\Omega^{\epsilon_n}(t),\boldsymbol{J}^{(l)}> - <\Omega^{\epsilon_m}(t),\boldsymbol{J}^{(l)}> |.
\end{align*}
by \eqref{lbb}.
Hence, for sufficiently large $n,m$, 
\[
\sup_{t \in [0,T]} | <\Omega^{\epsilon_n}(t),\boldsymbol{J}> - <\Omega^{\epsilon_m}(t),\boldsymbol{J}> |\le (1+K_0) \delta
\]
and so
$ < \Omega^{\epsilon_N}(\cdot),\boldsymbol{J}> $ converges uniformly.

In Appendix \ref{positivity}, we prove that for any $t \ge 0$, any limit of a weak-* convergent subsequence of $\{ \Omega^{\epsilon}(t) \}_{\epsilon}$ can be extended to a vector-valued finite positive measures on $\R \times \T$. 
The uniqueness of solutions of the equation (\ref{mboltzmann}) is shown in Appendix \ref{uniqueness}. 

Hence, noting that $\omega'(k) \partial_y\boldsymbol{J}, C\boldsymbol{J}, C'\boldsymbol{J} \in \mathbf{S}^{2}$ for any $\boldsymbol{J} \in \mathbf{S}^{2}$, by \eqref{timeevo} and the following lemma we conclude the proof of Theorem \ref{thm:main}. 

\begin{lemma}
For any $T>0$ and $\boldsymbol{J} \in \mathbf{S}^{2}$,
\begin{align*}
\lim_{\epsilon \to 0} |\int_{0}^{T} dt ~ <\Gamma^{\epsilon}(t),\boldsymbol{J}> | = 0 , \\
\lim_{\epsilon \to 0} |\int_{0}^{T} dt ~ <(\Gamma^{\epsilon})^{*}(t),\boldsymbol{J}> | = 0.
\end{align*}
\end{lemma}


\begin{proof}
By $(\ref{defofpsi})$ the time evolution of $\widehat{\Gamma_{i}^{\epsilon}}_{+}(t)(p,k), i =1,2$ is 
\begin{align*}
&\partial_{t} \widehat{\Gamma_{i}^{\epsilon}}_{+}(t)(p,k) \\
&= - \frac{\sqrt{-1}}{\epsilon} ({\omega}_i(k-\frac{\epsilon p}{2}) + \omega_{i^{*}}(k+\frac{\epsilon p}{2})) \widehat{\Gamma_{i}^{\epsilon}}_{+}(t)(p,k) \\
& ~ + \gamma (\widehat{\b}(k-\frac{\epsilon p}{2}) \theta_{i}^{2}(k-\frac{\epsilon p}{2}) + \widehat{\b}(k+\frac{\epsilon p}{2}) \theta_{i^{*}}^{2}(k+\frac{\epsilon p}{2}))\widehat{\Gamma_{i}^{\epsilon}}_{+}(t)(p,k) \\
& ~ + \gamma \widehat{\b}(k+\frac{\epsilon p}{2}) \theta_{i}(k+\frac{\epsilon p}{2}) \theta_{i^{*}}(k+\frac{\epsilon p}{2})\widehat{\Omega_{i}^{\epsilon}}_{-}(t)(p,k) \\
& ~ + \gamma \widehat{\b}(k-\frac{\epsilon p}{2}) \theta_{i}(k-\frac{\epsilon p}{2}) \theta_{i^{*}}(k-\frac{\epsilon p}{2})\widehat{\Omega_{i^{*}}^{\epsilon}}_{+}(t)(p,k) \\
& ~ + \gamma \theta_{i}(k-\frac{\epsilon p}{2}) \theta_{i^{*}}(k+\frac{\epsilon p}{2}) \int_{\T} dk' R_{\epsilon p}(k,k') \notag \\
& ~ ~ \times [\theta_{i}(k'-\frac{\epsilon p}{2}) \theta_{i}(k'+\frac{\epsilon p}{2}) \widehat{\Omega_{i}^{\epsilon}}_{+}(t)(p,k') + \theta_{i^{*}}(k'-\frac{\epsilon p}{2}) \theta_{i^{*}}(k'+\frac{\epsilon p}{2}) \widehat{\Omega_{i^{*}}^{\epsilon}}_{+}(t)(p,k')) \notag \\
& ~ ~ ~ + \theta_{i}(k'-\frac{\epsilon p}{2}) \theta_{i^{*}}(k'+\frac{\epsilon p}{2}) \widehat{\Gamma_{i}^{\epsilon}}_{+}(t)(p,k') + \theta_{i}(k'+\frac{\epsilon p}{2}) \theta_{i^{*}}(k'-\frac{\epsilon p}{2}) \widehat{\Gamma_{i}^{\epsilon}}_{-}(t)(p,k')] .
\end{align*}
Since $\widehat{\b}$, $\theta_{i}$ and ${\omega}_i$ are smooth on $\T$, the above term is rewritten as
\begin{align}\label{evoofWigner3}
&\partial_{t} \widehat{\Gamma_{i}^{\epsilon}}_{+}(t)(p,k) \notag \\
&= - \frac{\sqrt{-1}}{\epsilon} ({\omega}_i(k) + \omega_{i^{*}}(k)) \widehat{\Gamma_{i}^{\epsilon}}_{+}(t)(p,k) + \gamma (\widehat{\b}(k) \theta_{i}^{2}(k) + \widehat{\b}(k) \theta_{i^{*}}^{2}(k))\widehat{\Gamma_{i}^{\epsilon}}_{+}(t)(p,k)  \notag \\
& ~ + \gamma \widehat{\b}(k) \theta_{i}(k) \theta_{i^{*}}(k)\widehat{\Omega_{i}^{\epsilon}}_{-}(t)(p,k) + \gamma \widehat{\b}(k) \theta_{i}(k) \theta_{i^{*}}(k)\widehat{\Omega_{i^{*}}^{\epsilon}}_{+}(t)(p,k) \notag \\
& ~ + \gamma \theta_{i}(k) \theta_{i^{*}}(k) \int_{\T} dk' R(k,k') [\theta_{i}^{2}(k') \widehat{\Omega_{i}^{\epsilon}}_{+}(t)(p,k') + \theta_{i^{*}}^{2}(k') \widehat{\Omega_{i^{*}}^{\epsilon}}_{+}(t)(p,k')) \notag \\
& ~ ~ ~ + \theta_{i}(k') \theta_{i^{*}}(k') \widehat{\Gamma_{i}^{\epsilon}}_{+}(t)(p,k') + \theta_{i}(k') \theta_{i^{*}}(k') \widehat{\Gamma_{i}^{\epsilon}}_{-}(t)(p,k')] + \mathcal{R}_{i + 2}(p,k) 
\end{align}
for $i =1,2$ where $\mathcal{R}_{i}, i=3,4$ are the remainder terms which satisfy
\begin{align}\label{remainder2}
||\mathcal{R}_{i}(p,\cdot)||_{\mathbb{L}^{1}(\T)} \le C(T,B,\gamma,\a,K_{0}) |p| (1+\epsilon)
\end{align}
for all $p \in \R$. Hence, for any $\boldsymbol{J} \in \mathbf{S}^{2}$ and $i=1,2$,
\begin{align*}
\int_{\R} dp \int_{\T} dk ~ \mathcal{R}_{i+2}(p,k) \widehat{J_{i}}(p,k)^{*} = O_{\boldsymbol{J} }(1).
\end{align*}
Combining $(\ref{rewriteWigner})$, $(\ref{evoofWigner3})$ and $(\ref{remainder2})$ we have
\begin{align*}
&\partial_{t} <\Gamma^{\epsilon}(t),\boldsymbol{J}> \\
&= - \frac{\sqrt{-1}}{\epsilon} <\Gamma^{\epsilon},({\omega}_1+{\omega}_2)\boldsymbol{J} > + <\Omega^{\epsilon},R'\boldsymbol{J} > + <\Omega^{\epsilon},R'\boldsymbol{J}^{t}> + <\Gamma^{\epsilon},R''\boldsymbol{J}> \\
& ~ + <(\Gamma^{\epsilon})^{*},R''\boldsymbol{J}> + <\Omega^{\epsilon},\b'\boldsymbol{J}^{t}> + <(\Omega^{\epsilon})^{*},\b'\boldsymbol{J}> + <\Gamma^{\epsilon},\b \boldsymbol{J}> + ~ O_{\boldsymbol{J}}(1) , 
\end{align*}
where $\boldsymbol{J}^{t} = (J_{2},J_{1})$ and
\begin{align*}
& <(\Omega^{\epsilon})^{*}(t),\boldsymbol{J}> = \sum_{i=1,2} \int_{\R} dp \int_{\T} dk  ~ \widehat{\Omega_{i}^{\epsilon}}_{-}(t)(p,k) \widehat{J_{i}}(p,k)^{*}, \\
&\b'(k) = \theta_{1}(k) \theta_{2}(k) \b(k), \\
&(R'\boldsymbol{J})_{i}(p,k) = \int_{\T} dk' \theta_{i}(k)^2 R(k,k') \theta_{1}(k')\theta_{2}(k') J_{i}(p,k'), \\
&(R''\boldsymbol{J})_{i}(p,k) = \int_{\T} dk' \theta_{1}(k) \theta_{2}(k) R(k,k') \theta_{1}(k')\theta_{2}(k') J_{i}(p,k').
\end{align*}
Therefore, we have
\begin{align*}
\lim_{\epsilon \to 0} |\int_{0}^{T} dt ~ <\Gamma^{\epsilon}(t),({\omega}_1+{\omega}_2)\boldsymbol{J} > | = 0 
\end{align*}
for all $\boldsymbol{J} \in \mathbf{S}^{2}$. Since ${\omega}_1(k) + {\omega}_2(k) $ is uniformly bounded by positive constants from above and below, $({\omega}_1 + {\omega}_2)^{-1}\boldsymbol{J} \in \mathbf{S}^2$ for all $\boldsymbol{J}  \in \mathbf{S}^2$. Hence we conclude that 
\begin{align*}
\lim_{\epsilon \to 0} |\int_{0}^{T} dt ~ <\Gamma^{\epsilon}(t),\boldsymbol{J}> | = 0 
\end{align*}
for all $\boldsymbol{J} \in \mathbf{S}^{2}$. 

For $(\Gamma^{\epsilon})^{*}$ we can apply the same proof. 

\end{proof}

\section{Proof of Theorem \ref{thm:main2}.}\label{sec:proof2}

We use the Markov chain $(K(t),I(t))$ introduced in Section \ref{markov}. First note that for any $u_0 \in C^{\infty}_0(\R \times \T \times \{1,2\})$,
\[
u_N(y,k,i,t) = \mathbb{E}_{(k,i)}[u_{0}(Z_{N}(t),K(t),I(t))].
\]
where 
\begin{align*}
Z_{N}(t) = y + \frac{1}{2\pi N^{\frac{3}{5}}} \int_{0}^{t} ds ~ \omega'(K(s)).
\end{align*}
Then, by using the Fourier transform we can write
\begin{align*}
u_N(y,k,i,Nt) &= \mathbb{E}_{(k,i)}[u_{0}(Z_{N}(Nt),K(Nt),I(Nt))] \\
&= \sum_{x \in \Z} \int_{\R} dp \sum_{j = 1,2} \widetilde{u_{0}} (p,x,j) \mathbb{E}_{(k,i)}[e^{\sqrt{-1} p Z_{N}(Nt)} e^{\sqrt{-1} x K(Nt)} 1_{\{ I(Nt) = j \} } ] ,
\end{align*}
where $\widetilde{u}(p,x,i)$ is the Fourier transform of $u(y,k,i)$. 
Denote by $di$ the counting measure on $\{1,2 \}$. Let $P^{t} , t \ge 0$ be the semigroup generated by $\mathcal{L}$. Since $\frac{1}{2}dkdi$ is a reversible probability measure of the process $\{ ( K(t) , I(t) )\}_{t \ge 0}$ and $0$ is a simple eigenvalue for the generator $\mathcal{L}$, we have
\begin{align*}
\lim_{t \to \infty} ||P^{t}f||_{\mathbb{L}^{2}(\T \times \{ 1,2 \} )} = 0 
\end{align*}
for any $f \in \mathbb{L}^{2}(\T \times \{ 1,2 \} )$ satisfying $\int_{\T \times \{1,2 \}} dkdi f(k,i) = 0$ by the ergodicity and the reversibility (cf. Theorem 1.6.1, 1.6.3 and Exercise 4.7.2 of \cite{FOT}). Let $\{ m_{N}\}_{N \in \N  } $ be an increasing sequence of positive numbers such that
\begin{align*}
\lim_{N \to \infty} m_{N} = \infty , \\
\lim_{N \to \infty} m_{N} N^{-\frac{3}{5}} = 0.
\end{align*}
Then for any $t \ge 0 , p \in \R , x \in \Z$ and $j = 1,2$ 
\begin{align*}
& \left|\mathbb{E}_{(k,i)}[e^{\sqrt{-1} p Z_{N}(Nt)} e^{\sqrt{-1} x K(Nt)} 1_{\{ I(Nt) = j \} }] - \mathbb{E}_{(k,i)}[e^{\sqrt{-1} p Z_{N}(Nt - m_{N}t)} e^{\sqrt{-1} x K(Nt)} 1_{\{ I(Nt) = j \} }]  \right|\\
&\le \mathbb{E}_{(k,i)}[|1 - e^{\sqrt{-1} p (Z_{N}(Nt) - Z_{N}(Nt - m_{N}t))}|] \\
&\le \mathbb{E}_{(k,i)}[|p (Z_{N}(Nt) - Z_{N}(Nt - m_{N}t))|] 
\end{align*}
since $|1-e^{\sqrt{-1}a}| \le |a|$ for any $a \in \R$. The last expression converges to $0$ as $N \to \infty$ since \begin{align*}
\mathbb{E}_{(k,i)}[|p (Z_{N}(Nt) - Z_{N}(Nt - m_{N}t))|]  &=
 \mathbb{E}_{(k,i)}[|p \frac{1}{2\pi N^{\frac{3}{5}}} \int_{Nt - m_{N}t}^{Nt} ds ~ \omega'(K(s))|] \\
&\le \|\omega'\|_{\infty} t |p| m_{N} N^{-\frac{3}{5}} \to 0 
\end{align*}
where $ \|\omega'\|_{\infty}=\sup_{k} | \omega'(k)|$. By the Markov property
\begin{align*}
&\mathbb{E}_{(k,i)}[e^{\sqrt{-1} p Z_{N}(Nt - m_{N}t)} e^{\sqrt{-1} x K(Nt)} 1_{\{ I(Nt) = j \} }] \\
&= \mathbb{E}_{(k,i)}[e^{\sqrt{-1} p Z_{N}(Nt - m_{N}t)} \mathbb{E}_{(K(Nt - m_{N}t),I(Nt - m_{N}t))} [e^{\sqrt{-1} x K(m_{N}t)} 1_{\{ I(m_{N}t) = j \} }] ] .
\end{align*}
By the Schwarz's inequality,
\begin{align}\label{0inl^{2}}
& \big|\mathbb{E}_{(k,i)}[e^{\sqrt{-1} p Z_{N}(Nt - m_{N}t)} \mathbb{E}_{(K(Nt - m_{N}t),I(Nt - m_{N}t))} [e^{\sqrt{-1} x K(m_{N}t)} 1_{\{ I(m_{N}t) = j \} }] ] \notag \\
&  ~ - \mathbb{E}_{(k,i)}[e^{\sqrt{-1} p Z_{N}(Nt - m_{N}t)} \mathbb{E}_{(K(Nt - m_{N}t),I(Nt - m_{N}t))}[ \frac{1}{2} \int_{\T} dk' e^{\sqrt{-1} x k'}]] \big| \notag \\
&\le \mathbb{E}_{(k,i)}[|\mathbb{E}_{(K(Nt - m_{N}t),I(Nt - m_{N}t))}[e^{\sqrt{-1} x K(m_{N}t)} 1_{\{ I(m_{N}t) = j \} } - \frac{1}{2} \int_{\T} dk' e^{\sqrt{-1} x k'}]|^{2}]^{\frac{1}{2}} .
\end{align}
Let $g(k,i) = e^{\sqrt{-1} x k} 1_{\{ j \} }(i) - \frac{1}{2} \int_{\T} dk' e^{\sqrt{-1} x k'}$. Since $\frac{1}{2}dkdi$ is the reversible probability measure  we have
\begin{align*}
&\int_{\T \times \{1,2 \}} dkdi ~ \mathbb{E}_{(k,i)}[|\mathbb{E}_{(K(Nt - m_{N}t),I(Nt - m_{N}t))}[e^{\sqrt{-1} x K(m_{N}t)}1_{\{ I(m_{N}t) = j \} } - \frac{1}{2} \int_{\T} dk' e^{\sqrt{-1} x k'}]|^{2}] \\
&= ||P^{m_{N}t}g||_{\mathbb{L}^{2}(\T \times \{ 1,2 \} )}^{2} .
\end{align*}
Hence we conclude that \eqref{0inl^{2}} converges to 0 in $\mathbb{L}^{2}(\T \times \{ 1,2 \} )$ as $N \to \infty$ 
since $\int_{\T \times \{1,2 \}} dkdi ~ g(k,i) = 0$.

Summarizing the above and applying the dominated convergence theorem, we have 
\begin{align*}
&\lim_{N \to \infty} \int_{\T \times \{1,2 \}} dk di \sum_{x \in \Z} \int_{\R} dp \sum_{j = 1,2} | \widetilde{u_{0}} (p,x,j) | \\
&\times ~ |\mathbb{E}_{(k,i)}[e^{\sqrt{-1} p Z_{N}(Nt)} e^{\sqrt{-1} x K(Nt)} 1_{\{ I(Nt) = j \} } ] - \mathbb{E}_{(k,i)}[e^{\sqrt{-1} p Z_{N}(Nt - m_{N}t)} \frac{1}{2} \int_{\T} dk' e^{\sqrt{-1} x k'}] |^{2} \\
&= 0.
\end{align*}
Note that 
\begin{align*}
& \sum_{x \in \Z} \int_{\R} dp \sum_{j = 1,2} \widetilde{u_{0}} (p,x,j) \mathbb{E}_{(k,i)}[e^{\sqrt{-1} p Z_{N}(Nt - m_{N}t)} \frac{1}{2} \int_{\T} dk' e^{\sqrt{-1} x k'}] \\
&= \mathbb{E}_{(k,i)} [ \frac{1}{2} \int_{\T \times \{1,2 \}} dk' dj ~ u_{0}(Z_{N}(Nt - m_{N}t),k',j) ], \\
& =  \frac{1}{2} \mathbb{E}_{(k,i)} [\bar{u}_0(Z_{N}(Nt - m_{N}t))].
\end{align*}
By Theorem \ref{thm:main3}, $Z_{N}(Nt - m_{N}t)$ converges to a  L\'{e}vy  process starting from $y$ and generated by $D (-\Delta_{y})^{\frac{5}{6}}$ and so the last term converges to $\bar{u}(y,t)$ given in \eqref{boltzmann2} for $k \neq 0, ~ i=1,2$. 
Therefore,
\[
\frac{1}{2}\mathbb{E}_{(k,i)} [\bar{u}_0(Z_{N}(Nt - m_{N}t))] \to \frac{1}{2}\bar{u}(y,t) \quad a.e. 
\]
and by the dominated convergence theorem,
\begin{align*}
 & \limsup_{N \to \infty} \int_{\T \times \{1,2 \}} dkdi | u_N(y,k,i,Nt) - \frac{1}{2}\bar{u}(y,t) |^{2} \\
& \le \limsup_{N \to \infty}\int_{\T \times \{1,2 \}} dkdi | u_N(y,k,i,Nt) - \frac{1}{2} \mathbb{E}_{(k,i)} [\bar{u}_0(Z_{N}(Nt - m_{N}t))] |^{2}.
\end{align*}
Applying the Fourier transform, the last term is bounded from above by 
\begin{align*}
& \limsup_{N\to \infty} \left(\sum_{x \in \Z} \int_{\R} dp \sum_{j = 1,2} | \widetilde{u_{0}} (p,x,j) |\right) \int_{\T \times \{1,2 \}} dkdi \sum_{x \in \Z} \int_{\R} dp \sum_{j = 1,2} | \widetilde{u_{0}} (p,x,j) |  \\
& \quad \quad \times |\mathbb{E}_{(k,i)}[e^{\sqrt{-1} p Z_{N}(Nt)} e^{\sqrt{-1} x K(Nt)} 1_{\{ I(Nt) = j \} } ] - \mathbb{E}_{(k,i)}[e^{\sqrt{-1} p Z_{N}(Nt - m_{N}t)} \frac{1}{2} \int_{\T} dk' e^{\sqrt{-1} x k'}] |^{2}
\end{align*}
and so we complete the proof.

\section*{Acknowledgement}

KS was supported by JSPS Grants-in-Aid for Scientific Research No. JP17K05587 and No. JP16H02211. MS was supported by JSPS Grants-in-Aid for Scientific Research No. JP16KT0021.

\appendix


\section{Derivation of \eqref{formalk}}\label{derivation}

We only consider the time evolution of $\widehat{v}_{1}(k,t)$. By the same calculation one can get the time evolution of $\widehat{v}_{2}(k,t)$. From $(\ref{formalz})$ we have
\begin{align*}
d \widehat{v}_{1}(k,t) &= \sum_{x \in \Z} e^{-2 \pi \sqrt{-1} k x} d v_{1}(x,t) \\
&=( - \widehat{\a}(k) \widehat{q^1}(k,t) + B\widehat{v^2}(k,t)  + \epsilon \gamma \widehat{\b}(k) \widehat{v^1}(k,t) ) dt\\
& ~ + \sqrt{\epsilon \gamma} \sum_{x \in \Z} \sum_{z;|z-x|=1} e^{-2 \pi \sqrt{-1} k x} (Y_{x,z} v_{1}(x,t)) dw_{x,z}. 
\end{align*}
Now we compute the last term. By summation by parts we have
\begin{align*}
&-\sum_{x \in \Z} \sum_{z;|z-x|=1} e^{-2 \pi \sqrt{-1} k x} (Y_{x,z} v_{1}(x,t)) dw_{x,z} \\
&= \sum_{x \in \Z} \sum_{z \in \Z} h(z) v_{2}(x+z) e^{-2 \pi \sqrt{-1} k (x+z)} dw_{x,x+1} ,
\end{align*}
where $h: \Z \to \Z$ is defined as
\begin{align*}
h(z)=
\begin{cases}
 e^{2 \pi \sqrt{-1} k} -1 , \ z=1 \\
e^{-2 \pi \sqrt{-1} k} -1 , \ z= 0\\
 0 , \ z \neq 0,1.
\end{cases}
\end{align*}
By the change of variables, the last term is rewritten as 
\begin{align*}
&\sum_{x \in \Z} \sum_{z \in \Z} h(z) v_{2}(x+z) e^{-2 \pi \sqrt{-1} k (x+z)} dw_{x,x+1} \\
&= \sum_{x \in \Z} \sum_{x' \in \Z} h(x'-x) v_{2}(x') e^{-2 \pi \sqrt{-1} k x'} dw_{x,x+1} \\
&= \sum_{x,x' \in \Z} (\int_{\T} dk' e^{2 \pi \sqrt{-1} k' (x'-x)} \sum_{y \in \Z} e^{-2 \pi \sqrt{-1} k' y} h(y)) v_{2}(x') e^{-2 \pi \sqrt{-1} k x'} dw_{x,x+1} ,
\end{align*}
and 
\begin{align*}
\sum_{y \in \Z} e^{-2 \pi \sqrt{-1} k' y} h(y) &= e^{-2 \pi \sqrt{-1} k} -1 + e^{- 2 \pi \sqrt{-1} k'}(e^{2 \pi \sqrt{-1} k} -1) \\
&= (e^{-2 \pi \sqrt{-1} k'} - e^{-2 \pi \sqrt{-1} k})(e^{2 \pi \sqrt{-1} k} - 1) \\
&= r(k,k').
\end{align*}
Therefore we have $(\ref{formalk})$.

\section{Existence and uniqueness of the solution of \eqref{defofpsi}}\label{existence}

We follow the strategy of \cite{DZ} to show the existence by classical 
fixed point theorem. 

First we prepare some notations. We introduce a norm on $(\mathbb{L}^{2}(\T))^{2}$ defined as
\begin{align*}
||\mathbf{f}||_{(\mathbb{L}^{2}(\T))^{2}}^{2} = ||f_{1}||_{\mathbb{L}^{2}(\T)}^{2} + ||f_{2}||_{\mathbb{L}^{2}(\T)}^{2}
\end{align*}
for $\mathbf{f} = (f_{1},f_{2}) \in (\mathbb{L}^{2}(\T))^{2}$. Let $(E,\mathcal{F},\mathbb{P})$ be a probability space and $W$ be a cylindrical Wiener process on $(E,\mathcal{F},\mathbb{P})$. Fix $T>0$. Denote by $(\mathcal{H},||\cdot||_{\mathcal{H}})$ the Banach space of $(\mathbb{L}^{2}(\T))^{2}$-valued measurable processes $\mathbf{f}(k,t) , ~  k \in \T, t \in [0,T]$ such that
\begin{align*}
||\mathbf{f}||_{\mathcal{H}} = (\sup_{t \in [0,T]} \mathbb{E}[||\mathbf{f}(\cdot,t)||_{(\mathbb{L}^{2}(\T))^{2}}^{2}])^{\frac{1}{2}} < \infty ,
\end{align*}
where two processes are identified if they are $\mathbb{P} \times dt$ almost surely equal.

Next we rewrite \eqref{defofpsi} as
\begin{align*}
d \begin{pmatrix} \widehat{\psi_{1}}(k,t) \\ \widehat{\psi_{2}}(k,t) \end{pmatrix} &= A_{1}(\widehat{\boldsymbol{\psi}}(\cdot,t))(k) dt + A_{2}(k')(\widehat{\boldsymbol{\psi}}(\cdot,t))(k) W(dk',dt) ,
\end{align*}
where $A_{1}$ and $A_{2}(k') , k' \in \T$ are bounded linear operators on $(\mathbb{L}^{2}(\T))^{2}$ defined as
\begin{align*}
A_{1}(\mathbf{f})(k)&=\begin{pmatrix} \{ - \sqrt{-1} {\omega}_1(k) + \epsilon \gamma \b(k) \theta_{1}(k)^2 \} f_{1}(k) + \epsilon \gamma \b(k) \theta_{1}(k) \theta_{2}(k) f_{2}^{*}(k) \\ \epsilon \gamma \b(k) \theta_{1}(k) \theta_{2}(k) f_{1}^{*}(k) + \{ - \sqrt{-1} {\omega}_2(k) + \epsilon \gamma \b(k) \theta_{2}(k)^2 \} f_{2}(k) \end{pmatrix}, \\
A_{2}(k')(\mathbf{f})(k) &= \begin{pmatrix} r(k,k') (\theta_{1}(k-k') f_{1}(k-k') + \theta_{2}(k-k') f_{2}^*(k'-k) ) \\ r(k,k') (\theta_{1}(k-k') f_{1}^{*}(k'-k) + \theta_{2}(k-k') f_{2}(k-k') ) \end{pmatrix}
\end{align*}
for $\mathbf{f} = (f_{1},f_{2}) \in (\mathbb{L}^{2}(\T))^{2}$. Fix $\widehat{\boldsymbol{\psi}}_{0} \in (\mathbb{L}^{2}(\T))^{2}$. We define a functional $I:\mathcal{H} \to \mathcal{H}$ as
\begin{align*}
I(\mathbf{f})_t= \widehat{\boldsymbol{\psi}}_{0} +  \int_{[0,t]} A_{1}(\mathbf{f}(\cdot,s)) ds + \int_{[0,t]} A_{2}(k')(\mathbf{f}(\cdot,s)) W(dk',ds). 
\end{align*}
For sufficiently small $T > 0$, $I$ is contractive and so there exists the unique fixed point $\widehat{\psi} \in \mathcal{H}$ such that $I(\widehat{\boldsymbol{\psi}}) = \widehat{\boldsymbol{\psi}}$. In this way we can construct a solution on the time interval $[0,T]$, and then we can construct a solution on the time interval $[T,2T]$ by the same argument and so on.

Finally we check the uniqueness of the solution in the sense of the distribution. Suppose that $\mathbf{f}^{(1)}$ and $\mathbf{f}^{(2)} \in \mathcal{H}$ are two solutions of \eqref{defofpsi} with a same initial condition. By the Cauchy-Schwarz inequality and the boundedness of $A_{1}$ and $A_{2}$, we have
\begin{align*}
&\mathbb{E}[||\mathbf{f}^{(1)}(\cdot,t) - \mathbf{f}^{(2)}(\cdot,t)||_{(\mathbb{L}^{2}(\T))^{2}}^{2}] \\
&\le C(T) \int_{[0,t]} ds ~ \mathbb{E}[||\mathbf{f}^{(1)}(\cdot,s) - \mathbf{f}^{(2)}(\cdot,s)||_{(\mathbb{L}^{2}(\T))^{2}}^{2}]
\end{align*}
for all $t \ge 0$. By the Gronwall's inequality we have $\mathbb{E}[||\mathbf{f}^{(1)}(\cdot,t) - \mathbf{f}^{(2)}(\cdot,t)||_{(\mathbb{L}^{2}(\T))^{2}}^{2}] = 0$. 

\section{Conservation of the total energy}\label{conservation}

By \eqref{defofpsi}, the time evolution of $\int_{\T} dk ~ \mathbb{E}_{\epsilon} [ |\widehat{\psi_{i}}(k,t)|^{2} ] $ is given by
\begin{align*}
&\frac{d}{dt} \int_{\T} dk ~ \mathbb{E}_{\epsilon} [ |\widehat{\psi_{i}}(k,t)|^{2} ] \\
&= \int_{\T} dk ~ 2 \gamma \widehat{\b}(k) \theta_{i}^{2}(k) \mathbb{E}_{\epsilon} [ |\widehat{\psi_{i}}(k,t)|^{2} ]  + 2 \gamma \widehat{\b}(k) \theta_{i}(k) \theta_{i^{*}}(k) \Re(\mathbb{E}_{\epsilon} [\widehat{\psi_{i}}(k,t)\widehat{\psi_{i^{*}}}(k,t) ] ) \\
& ~ + \gamma \theta_{i}^{2}(k) \int_{\T} dk' R(k,k') \{ \theta_{i}^{2}(k') \mathbb{E}_{\epsilon} [ |\widehat{\psi_{i}}(k',t)|^{2} ] + \theta_{i^{*}}^{2}(k') \mathbb{E}_{\epsilon} [ |\widehat{\psi_{i^{*}}}(k',t)|^{2} ] \notag \\
& ~ ~ + 2 \theta_{i}(k') \theta_{i^{*}}(k') \Re(\mathbb{E}_{\epsilon} [ \widehat{\psi_{i}}(k,t)\widehat{\psi_{i^{*}}}(k,t)) ] ) \}
\end{align*}
where $\Re a$ is the real part of $a \in \C$.
Since $ \sum_{i=1,2} \theta_{i}(k)^2 = 1$ and $\int_{\T} dk' ~ R(k,k') = -2 \widehat{\b}(k)$, we have
\begin{align*}
&\frac{d}{dt} \sum_{i=1,2} \int_{\T} dk ~ \mathbb{E}_{\epsilon} [ |\widehat{\psi_{i}}(k,t)|^{2} ] \\
&= \int_{\T} dk ~ 2 \gamma \widehat{\b}(k) \sum_{i=1,2} \theta_{i}^{2}(k) \mathbb{E}_{\epsilon} [ |\widehat{\psi_{i}}(k,t)|^{2} ] + 4 \gamma \widehat{\b}(k) \theta_{1}(k) \theta_{2}(k) \Re(\mathbb{E}_{\epsilon} [ \widehat{\psi_{1}}(k,t)\widehat{\psi_{2}}(k,t) ] ) \\
& ~ + \gamma \int_{\T} dk' R(k,k') \{ \theta_{1}^{2}(k') \mathbb{E}_{\epsilon} [ |\widehat{\psi_{1}}(k',t)|^{2} ] + \theta_{2}^{2}(k') \mathbb{E}_{\epsilon} [ |\widehat{\psi_{2}}(k',t)|^{2} ] \notag \\
& ~ ~ + 2 \theta_{1}(k') \theta_{2}(k') \Re(\mathbb{E}_{\epsilon} [ \widehat{\psi_{1}}(k,t)\widehat{\psi_{2}}(k,t) ] ) \} \\
& = 0.
\end{align*}

\section{Uniqueness of solutions of the linear kinetic equation}\label{positivity}

\begin{lemma}\label{Riesz}
Let $\{ \Omega^{\epsilon_N}(t) \}_{N \in \N}$ be a convergent subsequence with its limit $\Omega(t)$. Then there exists a vector-valued finite positive measure $\boldsymbol{\mu}(t)=(\mu_1(t),\mu_2(t))$ such that
\begin{align*}
\int_{\R \times \T} \mu_i(t)(dy,dk) J_i(y,k)^* = <\Omega(t),J_i>, \quad i=1,2
\end{align*}
for all $\boldsymbol{J} \in \mathbf{S}^2$. 
\end{lemma}

\begin{proof}
Let $\Omega_{1}(t) \in \mathbf{S}'$ as $<\Omega_{1}(t), J>:= <\Omega(t), \boldsymbol{J}>$ for $\boldsymbol{J}=(J,0)$, $J \in \mathbf{S}$. 
First we show that $\Omega_1(\cdot)$ is multiplicatively positive, that is,
\begin{align*}
<\Omega_{1}(t),|J|^{2}> ~ \ge 0 
\end{align*}
for all $t \ge 0$ and $J \in \mathbf{S}$. Fix $t \ge 0$ and $J \in \mathbf{S}$. Since $J$ is smooth,
\begin{align*}
J(\frac{\epsilon}{2}(x+x'),k) - J(\epsilon x,k) &= \frac{\epsilon}{2} \int_{0}^{1} dr ~ (x' - x) \partial_{y} J(\epsilon x + r \frac{\epsilon}{2}(x' - x),k) 
\end{align*}
for all $x,x' \in \Z$ and so 
\begin{align*}
&\left| \int_{\T} dk e^{2\pi \sqrt{-1} (x'-x) k} \left(J(\frac{\epsilon}{2} (x+x'),k)J(\frac{\epsilon}{2} (x+x'),k)^* - J(\epsilon x,k) J(\frac{\epsilon}{2} (x+x'),k)^{*} \right) \right| \\
&= \left| \frac{\epsilon}{2}(x'-x) \int_{\T} dk e^{2\pi \sqrt{-1} (x'-x) k} J(\frac{\epsilon}{2} (x+x'),k)^{*} \int_{0}^{1} dr \partial_{y} J(\epsilon x + r \frac{\epsilon}{2} (x' - x) , k) \right|.
\end{align*}
By repeating the integration by parts we have
\begin{align*}
& \left|  \int_{\T} dk e^{2\pi \sqrt{-1} (x'-x) k} J(\frac{\epsilon}{2} (x+x'),k)^{*} \int_{0}^{1} dr \partial_{y} J(\epsilon x + r \frac{\epsilon}{2} (x' - x) , k) \right| \\
& = \left | \int_{\T} dk \left(\frac{1}{2 \pi \sqrt{-1} (x' - x)}\right)^{3} e^{2\pi \sqrt{-1} (x'-x) k}  \partial_{k}^{3}[ J(\frac{\epsilon}{2} (x+x'),k)^{*} \int_{0}^{1} dr \partial_{y} J(\epsilon x + r \frac{\epsilon}{2} (x' - x) , k) ] \right| \\
&\le \frac{1}{8\pi^3|x - x'|^{3}}  \int_{\T} dk  \ | \partial_{k}^{3}[ J(\frac{\epsilon}{2} (x+x'),k)^{*} \int_{0}^{1} dr \partial_{y} J(\epsilon x + r \frac{\epsilon}{2} (x' - x) , k) ] |. 
\end{align*}
Hence, we have
\begin{align*}
& \left| \int_{\T} dk e^{2\pi \sqrt{-1} (x'-x) k} \left(J(\frac{\epsilon}{2} (x+x'),k)J(\frac{\epsilon}{2} (x+x'),k)^* - J(\epsilon x,k) J(\frac{\epsilon}{2} (x+x'),k)^{*} \right) \right| \\
& \le \frac{\epsilon}{16\pi^3|x - x'|^{2}}  \int_{\T} dk  \ | \partial_{k}^{3}[ J(\frac{\epsilon}{2} (x+x'),k)^{*} \int_{0}^{1} dr \partial_{y} J(\epsilon x + r \frac{\epsilon}{2} (x' - x) , k) ] | \\
& \le \frac{1}{|x - x'|^{2}}O_{J}(\epsilon)
\end{align*}
for all $x \neq x' \in \Z$. In the same way, we can show that
\begin{align*}
\left| \int_{\T} dk e^{2\pi \sqrt{-1} (x'-x) k} \left(J(\epsilon x,k) J(\frac{\epsilon}{2} (x+x'),k)^{*} - J(\epsilon x,k) J(\epsilon x',k)^{*} \right) \right| \le \frac{1}{|x - x'|^{2}}O_{J}(\epsilon).
\end{align*}

On the other hand we have
\begin{align*}
&\frac{\epsilon}{2} \sum_{x,x' \in \Z} <\psi_{1}(x',\frac{t}{\epsilon})^{*} \psi_{1}(x,\frac{t}{\epsilon})>_{\epsilon} \int_{\T} dk ~ e^{2\pi \sqrt{-1} (x'-x) k} J(\epsilon x,k) J(\epsilon x',k)^{*} \\
&= \frac{\epsilon}{2} \int_{\T} dk ~ <|\sum_{x \in \Z} e^{- 2\pi \sqrt{-1} x k} \psi(x,\frac{t}{\epsilon}) J(\epsilon x,k)|^{2} >_{\epsilon} ~ \ge 0.
\end{align*}
Since
\begin{align*}
& \left | \int_{\T} dk e^{2\pi \sqrt{-1} (x'-x) k} |J(\frac{\epsilon}{2} (x+x'),k)|^{2} - J(\epsilon x,k) J(\epsilon x',k)^{*} \right| \\
&\le \left| \int_{\T} dk e^{2\pi \sqrt{-1} (x'-x) k} \left(J(\frac{\epsilon}{2} (x+x'),k)J(\frac{\epsilon}{2} (x+x'),k)^* - J(\epsilon x,k) J(\frac{\epsilon}{2} (x+x'),k)^{*} \right) \right| \\
& ~ + \left| \int_{\T} dk e^{2\pi \sqrt{-1} (x'-x) k} \left(J(\epsilon x,k) J(\frac{\epsilon}{2} (x+x'),k)^{*} - J(\epsilon x,k) J(\epsilon x',k)^{*} \right) \right|,
\end{align*}
combining the above calculations we have 
\begin{align*}
<\Omega_{1}^{\epsilon}(t),|J|^{2}> = \frac{\epsilon}{2} \int_{\T} dk ~ <|\sum_{x \in \Z} e^{- 2\pi \sqrt{-1} x k} \psi_{1}(x,\frac{t}{\epsilon}) J(\epsilon x,k)|^{2} >_{\epsilon} + O_{J}(\epsilon).
\end{align*}
Therefore $\Omega_{1}(\cdot)$ is multiplicatively positive. 

Next we show that $\Omega_{1}(\cdot)$ is positive, that is, 
\begin{align*}
<\Omega_{1}(t),J> ~ \ge 0
\end{align*}
for all $t \ge 0$ and  $J \in \mathbf{S} , J \ge 0$. Since $\{ J \in \mathbf{S} ; J \in C_{0}^{\infty}(\R \times \T) , J \ge 0 \}$ is a dense subset of $\{ J \in \mathbf{S} ; J \ge 0 \} $, it is sufficient to show the positivity on $C_{0}^{\infty}(\R \times \T)$. Fix $t \ge 0$ and a positive function $J \in C_{0}^{\infty}(\R \times \T)$. There exists a positive constant $M > 0$ such that the support of $J$ is a subset of $[-M,M] \times \T$. Let $a(y) \in C_{0}^{\infty}(\R)$ be a function such that $a(y) = 1 , ~ y \in [-M,M]$. Define $J^{(m)}(y,k) \in C_{0}^{\infty}(\R \times \T) , ~ m \in \N$ as
\begin{align*}
J^{(m)}(y,k) = a(y) \sqrt{J(y,k) + \frac{1}{m} }.
\end{align*}
Then the sequence $\{ |J^{(m)}|^{2} , ~ m \in \N \}$ converges to $J(y,k)$ in the topology of $C_{0}^{\infty}(\R \times \T)$. Since the embedding of the space $C_{0}^{\infty}(\R \times \T)$ into the space $\mathbf{S}$ is continuous, $\{ |J^{(m)}|^{2} , ~ m \in \N \}$ also converges to $J(y,k)$ in the topology of $\mathbf{S}$. By the continuity of $\Omega_{1}(t)$, we have
\begin{align*}
<\Omega_{1}(t),J> = \lim_{m \to \infty} <\Omega_{1}(t),|J^{(m)}|^{2}> \ \ge 0.
\end{align*}
Therefore $\Omega_{1}(\cdot)$ is positive. 

In the same way we can show that $\Omega_{2}(\cdot)$ is also positive.

By the usual method, for example see Lemma 1 in Chapter 2 of \cite{GV}, we can extend the domain of $\Omega_{i}(\cdot) , ~ i=1,2$ to the space $C_{0}(\R \times \T)$. By the Riesz representation theorem there exists a finite positive measure $\mu_{i}(\cdot) , ~ i=1,2$ such that
\begin{align*}
<\Omega_{i}(t),J> ~ = \int_{\R \times \T} \mu_{i}(t)(dy,dk) ~ J(y,k) , ~ i=1,2
\end{align*}
for all $t \ge 0$ and $J \in C_{0}(\R \times \T)$. By the linearity and the definition of $\Omega_{i}^{\epsilon} ( \cdot)$, 
\begin{align*}
<\Omega_{i}(t),J> ~ = \int_{\R \times \T} \mu_{i}(t)(dy,dk) ~ J(y,k)^* , ~ i=1,2
\end{align*}
for all $J \in \mathbf{S}$.
\end{proof}

\section{Uniqueness of the solution of the Boltzmann equation}\label{uniqueness}

Suppose that a vector-valued finite positive measure $\boldsymbol{\mu}(t)$ is a solution of the Boltzmann equation \eqref{boltzmann}. Then $\tilde{\boldsymbol{\mu}}(t)(dy,dk) := \boldsymbol{\mu}(t)(dy + \frac{1}{2\pi} \omega'(k)t,dk)$ is a solution of the following space-homogeneous Boltzmann equation
\begin{align*}
\partial_{t} \int d\tilde{\boldsymbol{\mu}}(t) \cdot \boldsymbol{J} = \int d\tilde{\boldsymbol{\mu}}(t) \cdot (C\boldsymbol{J})
\end{align*}
where 
\begin{align*}
\int d\tilde{\boldsymbol{\mu}}(t) \cdot \boldsymbol{J} &= \int \boldsymbol{\mu}(t)(dy + \frac{1}{2\pi} \omega'(k)t,dk) \cdot \boldsymbol{J} \\ 
&:= \int \boldsymbol{\mu}(t)(dy,dk) \cdot \boldsymbol{J}(y - \frac{1}{2\pi} \omega'(k)t,k). 
\end{align*}
Conversely, if $\tilde{\boldsymbol{\mu}}(t)$ is a solution of the space-homogeneous Boltzmann equation, then $\boldsymbol{\mu}(t)(dy,dk) := \tilde{\boldsymbol{\mu}}(t)(dy - \frac{1}{2\pi} \omega'(k)t,dk)$ is a solution of the Boltzmann equation \eqref{boltzmann}. 
Therefore, it is sufficient to show the uniqueness of the solution for the space-homogeneous Boltzmann equation. 

Suppose that $J^{1}(y,k) = f^{\lambda,y^{*},r}(y)G(k) , J^{2}(y,k) = 0$, where 
\begin{align*}
f^{\lambda,y^{*},r}(y) &=\exp \left( - \frac{\lambda}{r^{2} - |y -y^{*}|^{2}} \right) 1_{B(y^{*},r)}(y), \\
B(y^{*},r) &= \{ y \in \R \ ; \ |y - y^{*}| < r \},  
\end{align*}
$y^{*} \in \R$ , $r > 0$ and $G(\cdot) \in C^{\infty}(\T)$. Note that $f^{\lambda,y^{*},r} \in C^{\infty}_{0}(\R)$, $\|f^{\lambda,y^{*},r} \|_{\infty} \le 1$ and 
\begin{align*}
\lim_{\lambda \to 0} f^{\lambda,y^{*},r}(y) = 1_{B(y^{*},r)}(y).
\end{align*}
Let $\boldsymbol{\mu}(t), \boldsymbol{\nu}(t)$ be solutions of the space-homogeneous Boltzmann equation with a same initial condition. Then \begin{align*}
&|\int d\boldsymbol{\mu}(t) \cdot \boldsymbol{J} - \int d\boldsymbol{\nu}(t) \cdot \boldsymbol{J} | \\
& = |\int d\boldsymbol{\mu}(t) \cdot f^{\lambda,y^{*},r}(y) \boldsymbol{G} - \int d\boldsymbol{\nu}(t) \cdot f^{\lambda,y^{*},r}(y) \boldsymbol{G}  |\\
&\le  \int_{0}^{t} ds \left| \int d(\boldsymbol{\mu}(s) - \boldsymbol{\nu}(s)) \cdot (f^{\lambda,y^{*},r}(y) C\mathbf{G}(k)) \right|
\end{align*}
where $\mathbf{G}=(G(k),0)$.
By taking the limit $\lambda \to 0$, we have
\begin{align*}
&|\int_{\T} G(k) d(\mu_{1}(t)(B(y^{*},r),dk) - \nu_{1}(t)(B(y^{*},r),dk))| \\
&\le \int_{0}^{t} ds |\int_{\T} d(\mu(s)(B(y^{*},r),dk) - \nu(s)(B(y^{*},r),dk)) \cdot (C\mathbf{G})| \\
&\le \int_{0}^{t} ds \sum_{i=1,2} |\int_{\T} d(\mu_{i}(s)(B(y^{*},r),dk) - \nu_{i}(s)(B(y^{*},r),dk)) (C\mathbf{G})_{i}| \\
&\le 32 \sup_{k}|G(k)| \int_{0}^{t} ds \sum_{i=1,2} ||\mu_{i}(s)(B(y^{*},r),dk) - \nu_{i}(s)(B(y^{*},r),dk)|| \\
\end{align*}
where $||\cdot||$ is the total variation for a bounded signed measure on $\T$. Hence,
\begin{align*}
& ||\mu_{1}(t)(B(y^{*},r),dk) - \nu_{1}(t)(B(y^{*},r),dk)|| \\
& \quad \quad \le 32 \int_{0}^{t} ds \sum_{i=1,2} ||\mu_{i}(s)(B(y^{*},r),dk) - \nu_{i}(s)(B(y^{*},r),dk)||.
\end{align*}
By the same proof, we have
\begin{align*}
&||\mu_{2}(t)(B(y^{*},r),dk) - \nu_{2}(t)(B(y^{*},r),dk)|| \\
& \quad \le 32 \int_{0}^{t} ds \sum_{i=1,2} ||\mu_{i}(s)(B(y^{*},r),dk) - \nu_{i}(s)(B(y^{*},r),dk)||. \\
&\therefore \sum_{i=1,2} ||\mu_{i}(t)(B(y^{*},r),dk) - \nu_{i}(t)(B(y^{*},r),dk)|| \\
& \quad \quad \le 64  \int_{0}^{t} ds \sum_{i=1,2} ||\mu_{i}(s)(B(y^{*},r),dk) - \nu_{i}(s)(B(y^{*},r),dk)|| .  
\end{align*}
Therefore $\mu_{i}(t)(B(y^{*},r),dk) = \nu_{i}(t)(B(y^{*},r),dk)$ on $\T$ for any ball $B(y^{*},r) \subset \R$, which concludes $\boldsymbol{\mu}(t)=\boldsymbol{\nu}(t)$ for any $t \ge 0$.

\section{Derivation of $(\ref{evoofWigner})$}\label{evolution}

We only consider the time evolution of $\widehat{\Omega_{1}^{\epsilon}}_{+}(t)(p,k)$. By the same calculation we can obtain the time evolution of $\widehat{\Omega_{2}^{\epsilon}}_{+}(t)(p,k)$. From $(\ref{defofpsi})$ we have
\begin{align*}
&\partial_{t} \widehat{\Omega_{1}^{\epsilon}}_{+}(t)(p,k) \notag \\
&= - \frac{\sqrt{-1}}{\epsilon}\left({\omega}_1(k+\frac{\epsilon p}{2}) - {\omega}_1(k-\frac{\epsilon p}{2})\right) \widehat{\Omega_{1}^{\epsilon}}_{+}(t)(p,k) \notag \\
& ~ + \gamma \left(\b(k+\frac{\epsilon p}{2}) \theta_{1}(k+\frac{\epsilon p}{2})^2 + \b(k-\frac{\epsilon p}{2}) \theta_{1}(k-\frac{\epsilon p}{2})^2 \right) \widehat{\Omega_{1}^{\epsilon}}_{+}(t)(p,k) \notag \\
& ~ + \gamma \b(k+\frac{\epsilon p}{2}) \theta_{1}(k+\frac{\epsilon p}{2}) \theta_{2}(k+\frac{\epsilon p}{2}) \widehat{\Gamma_{2}^{\epsilon}}_{-}(t)(p,k) \notag \\
& ~ + \gamma \b(k-\frac{\epsilon p}{2}) \theta_{1}(k-\frac{\epsilon p}{2}) \theta_{2}(k-\frac{\epsilon p}{2})\widehat{\Gamma_{2}^{\epsilon}}_{+}(t)(p,k) \notag \\
& + \gamma \theta_{1}(k-\frac{\epsilon p}{2}) \theta_{1}(k+\frac{\epsilon p}{2}) \int_{\T} dk' r(k-\frac{\epsilon p}{2},k')^{*} r(k+\frac{\epsilon p}{2},k') \notag \\
& ~ ~ \times [\theta_{1}(k-k'-\frac{\epsilon p}{2}) \theta_{1}(k-k'+\frac{\epsilon p}{2}) \widehat{\Omega_{1}^{\epsilon}}_{+}(t)(p,k-k') \\
& ~ ~ ~ + \theta_{2}(k-k'-\frac{\epsilon p}{2}) \theta_{2}(k-k'+\frac{\epsilon p}{2}) \widehat{\Omega_{2}^{\epsilon}}_{+}(t)(p,k-k')) \notag \\
& ~ ~ ~ + \theta_{1}(k-k'-\frac{\epsilon p}{2}) \theta_{2}(k-k'+\frac{\epsilon p}{2}) \widehat{\Gamma_{2}^{\epsilon}}_{-}(t)(p,k-k') \\
& ~ ~ ~ + \theta_{1}(k-k'+\frac{\epsilon p}{2}) \theta_{2}(k-k'-\frac{\epsilon p}{2}) \widehat{\Gamma_{2}^{\epsilon}}_{+}(t)(p,k-k')]. \notag 
\end{align*}
By the change of variables $k-k' \to k'$, the last integral is rewritten as 
\begin{align*}
&\int_{\T} dk' r(k-\frac{\epsilon p}{2},k-k')^{*} r(k+\frac{\epsilon p}{2},k-k') \notag \\
& ~ ~ \times [\theta_{1}(k'-\frac{\epsilon p}{2}) \theta_{1}(k'+\frac{\epsilon p}{2}) \widehat{\Omega_{1}^{\epsilon}}_{+}(t)(p,k') + \theta_{2}(k'-\frac{\epsilon p}{2}) \theta_{2}(k'+\frac{\epsilon p}{2}) \widehat{\Omega_{2}^{\epsilon}}_{+}(t)(p,k')) \notag \\
& ~ ~ ~ + \theta_{1}(k'-\frac{\epsilon p}{2}) \theta_{2}(k'+\frac{\epsilon p}{2}) \widehat{\Gamma_{2}^{\epsilon}}_{-}(t)(p,k') + \theta_{1}(k'+\frac{\epsilon p}{2}) \theta_{2}(k'-\frac{\epsilon p}{2}) \widehat{\Gamma_{2}^{\epsilon}}_{+}(t)(p,k')]. \notag
\end{align*}
Hence, it is sufficient to show that $r(k-\frac{\epsilon p}{2},k-k')^{*} r(k+\frac{\epsilon p}{2},k-k') = R_{\epsilon p}(k,k')$. By the following direct calculations
\begin{align*}
&r(k-\frac{\epsilon p}{2},k-k')^{*} r(k+\frac{\epsilon p}{2},k-k') \\
&= (e^{2 \pi \sqrt{-1} (k-k')} - e^{2 \pi \sqrt{-1} (k-\frac{\epsilon p}{2})})(e^{- 2 \pi \sqrt{-1} (k-\frac{\epsilon p}{2})} - 1) \\
& ~ \times (e^{- 2 \pi \sqrt{-1} (k-k')} - e^{- 2 \pi \sqrt{-1} (k+\frac{\epsilon p}{2})})(e^{2 \pi \sqrt{-1} (k+\frac{\epsilon p}{2})} - 1) \\
&= (1 - e^{- \pi \sqrt{-1} \epsilon p}(e^{2 \pi \sqrt{-1} k'} + e^{-2 \pi \sqrt{-1} k'}) + e^{- 2 \pi \sqrt{-1} \epsilon p}) \\
& ~ \times (1 - e^{ \pi \sqrt{-1} \epsilon p}(e^{2 \pi \sqrt{-1} k} + e^{-2 \pi \sqrt{-1} k}) + e^{ 2 \pi \sqrt{-1} \epsilon p}) \\
&= (e^{ \pi \sqrt{-1} \epsilon p} - (e^{2 \pi \sqrt{-1} k'} + e^{-2 \pi \sqrt{-1} k'}) + e^{- \pi \sqrt{-1} \epsilon p}) \\
& ~ \times (e^{- \pi \sqrt{-1} \epsilon p} - (e^{2 \pi \sqrt{-1} k} + e^{-2 \pi \sqrt{-1} k}) + e^{\pi \sqrt{-1} \epsilon p}) 
\end{align*}
and
\begin{align*}
&e^{- \pi \sqrt{-1} \epsilon p} - (e^{2 \pi \sqrt{-1} k} + e^{-2 \pi \sqrt{-1} k}) + e^{\pi \sqrt{-1} \epsilon p} \\
&= 2\cos{\pi \epsilon p} - 2\cos{2 \pi k} \\
&= 4 \sin{(k+\frac{\epsilon p}{2})}\sin{(k-\frac{\epsilon p}{2})} ,
\end{align*}
we can verify the equation $r(k-\frac{\epsilon p}{2},k-k')^{*} r(k+\frac{\epsilon p}{2},k-k') = R_{\epsilon p}(k,k')$.


\begin{thebibliography}{15}
\bibitem{BBO} {\sc G. \ Basile, C. \ Bernardin, S. \ Olla} : {\em Thermal Conductivity for a Momentum Conservative Model}.
Commun. \ Math. \ Phys. \textbf{287}, 67-98 (2009)
\bibitem{BOS} {\sc G. \ Basile, S. \ Olla, H. \ Spohn} : {\em Energy transport in stochastically perturbed lattice dynamics}.
Arch. \ Ration. \ Mech. \textbf{195}, 171--203 (2009)
\bibitem{BGJ} {\sc C. \ Bernardin, P. \ Gon\c{c}alves and M. \ Jara}, {\em $3/4$-fractional superdiffusion in a system of harmonic oscillators perturbed by a conservative noise}, Arch. Rational Mech. Anal. \textbf{220} (2016), 505--542. 
\bibitem{DZ} {\sc G. \ Da Prato, J.\ Zabczyk} : {\em Stochastic equations in infinite dimensions}.
Cambridge University Press, Cambridge (1992)
\bibitem{D} {\sc A. \ Dhar} : {\em Heat transport in low-dimensional systems}.
Adv. \ Phys. \textbf{57}(5), 457--537 (2008)
\bibitem{FFL} {\sc J. \ Fritz, T. \ Funaki, J.L. \ Lebowitz} : {\em Stationary states of random Hamiltonian systems},
Probab. Theory relat. Fields \textbf{99}, 211--236 (1994)

\bibitem{FOT} {\sc Fukushima, M., Oshima, Y. and Takeda, M.} : {\em  Dirichlet Forms and Symmetric Markov Processes}, Berlin, Boston: De Gruyter, 2nd ed. (2010) 
\bibitem{GV} {\sc I. \ M. \ Gelfand, N. \ Ya. \ Vilenkin} : {\em Generalized Functions volume 4}.
Academic Press, New York (1964)
\bibitem{KJO} {\sc M. \ Jara, T. \ Komorowski,  S. \ Olla} : {\em A limit theorem for an additive functionals of Markov chains}.
Ann. \ Appl. \ Probab. \textbf{19}, 2270--2230 (2009)
\bibitem{JKO} {\sc M. \ Jara, T. \ Komorowski, S. \ Olla} : {\em Superdiffusion of Energy in a Chain of Harmonic Oscillators with Noise}. Commun. \ Math. \ Phys. \textbf{339}, 407--453 (2015)
\bibitem{KO} {\sc T. \ Komorowski and S. \  Olla}, {\em Diffusive Propagation of Energy in a Non-acoustic Chain}, Arch. Rational Mech. Anal., \textbf{223} (2017), 95--139 

\bibitem{Ls} {\em Thermal Transport in Low Dimensions : From Statistical Physics to Nanoscale Heat Transfer}.
edited by S. Lepri (Springer, New York), (2016)
\bibitem{LLP} {\sc S. \ Lepri, R. \ Livi, A. \ Politi} : {\em Thermal conduction in classical low-dimensional lattices}.
Phys. Rep. \textbf{377}(1) 1--80 (2003)
\bibitem{SS} {\sc K. \ Saito, M. \ Sasada} : {\em Thermal conductivity for a stochastic dynamics in a magnetic field}.
Commun. \ Math. \ Phys. \textbf{361}, 951--995 (2018)
\bibitem{S} {\sc H. \ Spohn} : {\em Nonlinear fluctuating hydrodynamics for anharmonic chains}. J. Stat. Phys. \textbf{154}(5), 1191-1227 (2014)
\bibitem{TSS} {\sc S.  \ Tamaki,  M.  \ Sasada,  K.  \ Saito} : {\em Heat Transport via Low-Dimensional Systems with Broken Time-Reversal Symmetry}. 
Phys.  Rev.  Lett.  \textbf{119} (2017)
\end{thebibliography}
\end{document}